\documentclass[]{sn/interact}

\usepackage{epstopdf}
\usepackage[caption=false]{subfig}

\usepackage[numbers,sort&compress]{natbib}
\bibpunct[, ]{[}{]}{,}{n}{,}{,}
\renewcommand\bibfont{\fontsize{10}{12}\selectfont}

\theoremstyle{plain}
\newtheorem{theorem}{Theorem}[section]
\newtheorem{lemma}[theorem]{Lemma}

\theoremstyle{definition}

\theoremstyle{remark}
\newtheorem{remark}{Remark}[section]

\usepackage{hyperref}
\usepackage{graphicx}
\usepackage{dsfont}
\usepackage{mathtools}
\usepackage{wrapfig}
\usepackage{amssymb, amsmath, latexsym}
\usepackage{nicefrac}
\usepackage{hhline}
\usepackage{multirow}
\usepackage{comment}
\usepackage{pifont}

\usepackage[colorinlistoftodos,bordercolor=orange,backgroundcolor=orange!20,linecolor=orange,textsize=scriptsize]{todonotes}
\usepackage{algorithm}\usepackage{algpseudocode}
\usepackage[symbol]{footmisc}

\usepackage{tikz}
\usetikzlibrary{matrix, positioning, fit}

\def \R {\mathbb R}

\def\wtgg{\widetilde{\nabla}}
\def\wha{\widehat{\alpha}}
\def\whdelta{\widehat{\delta}}
\def\whL{\widehat{L}}

\newcommand{\proofref}[1]{\hyperref[proof:#1]{\ensuremath{\downarrow}}}

\begin{document}


\title{Adaptive First-Order Methods with Composite Gradient Error\thanks{}}

\author{
\name{
Artem Vasin \textsuperscript{a} and
Alexander Tyurin \textsuperscript{b, c} and
Alexander Gasnikov \textsuperscript{d, a, e, f}
}
\affil{
\textsuperscript{a}Moscow Independent Research Institute of Artificial Intelligence, Moscow, Russia.
\textsuperscript{b}AXXX, Moscow, Russia.
\textsuperscript{c}Applied AI Institute, Moscow, Russia.
\textsuperscript{d}Innopolis University, Innopolis, Russia.
\textsuperscript{e}Steklov Mathematical Institute of the Russian Academy of Sciences, Moscow, Russia.
\textsuperscript{f}Institute for Information Transmission Problems, Moscow, Russia.
}
}

\date{Received: date / Accepted: date}
\maketitle              
\begin{abstract}
We study first-order optimization with inexact gradients whose errors combine a component proportional to the true gradient norm and a uniformly bounded absolute component. We develop an adaptive gradient descent scheme that selects estimates of the smoothness constant and both error parameters through a computable backtracking condition using exact function values. Its analysis yields convergence for smooth objectives and functions satisfying Polyak–$\L$ojasiewicz condition. For strongly convex objectives, we propose a composite-error intermediate gradient method, that adapts to unknown smoothness while using known error bounds and a strong convexity parameter $\mu$. With fixed smoothness $L$, acceleration is retained when $\alpha=\mathcal{O}\left(\left(\mu / L \right)^{\nicefrac{1}{4}}\right)$, larger admissible errors yield intermediate rates. Throughout the range $0 \leq \alpha \leq \nicefrac{1}{3}$, the absolute error bound $\delta$ contributes at most $\mathcal{O} \left(\delta^2/\mu \right)$ to the objective residual, without additional amplification by the condition number.
\end{abstract}


\section{Introduction}\label{Introduction}

Inexact gradient information affects both the progress of an optimization method and the accuracy of its output. Relative errors scale with the magnitude of the true gradient and therefore diminish as the gradient vanishes, whereas an absolute error component may persist near a minimizer. When the smoothness constant or the error parameters are unknown, the algorithm must also determine suitable step sizes from the information available during the iteration. The purpose of this paper is to study these issues within a common deterministic error model and to develop adaptive gradient methods with explicit convergence guarantees.

We consider the unconstrained minimization problem:
\begin{equation}\label{optim}
    \min_{x\in\mathbb{R}^n} f(x),
\end{equation}
where $f:\mathbb{R}^n\to\mathbb{R}$ is differentiable. We assume that the  minimum is attained and write $f^*=\min\limits_{x\in\mathbb{R}^n}f(x)$. Given an initial point $x^0$, let $x^*$ be a minimizer closest to $x^0$ in the Euclidean norm, and set $R=\|x^0-x^*\|_2$. Throughout, $\langle\cdot,\cdot\rangle$ and $\|\cdot\|_2$ denote the Euclidean inner product and norm, respectively. For a differentiable function, convexity is characterized by
\begin{equation}\label{convexity}
    (\forall x,y\in\mathbb{R}^n) \qquad f(x)+ \langle \nabla f(x),y - x \rangle \leqslant f(y).
\end{equation}
We express $L$-smoothness, with $L>0$, through the quadratic upper model
\begin{equation}\label{smooth cond}
    (\forall x,y \in\mathbb{R}^n) \qquad f(y) \leqslant f(x) + \langle \nabla f(x), y - x \rangle
    + \frac{L}{2} \|x - y\|_2^2.
\end{equation}
For the analysis of the intermediate method, we additionally assume that
$f$ is $\mu$-strongly convex, with $\mu>0$:
\begin{equation}\label{strong conv}
    (\forall x,y \in\mathbb{R}^n)\qquad f(x) + \langle \nabla f(x), y - x \rangle +\frac{\mu}{2} \|x - y\|_2^2 \leqslant f(y).
\end{equation}
In this case the minimizer is unique. Some of our gradient descent results
instead use the Polyak--{\L}ojasiewicz (PL) condition with parameter
$\mu>0$:
\begin{equation}\label{PL cond}
    (\forall x\in\mathbb{R}^n) \qquad \frac{1}{2\mu}\|\nabla f(x)\|_2^2 \geqslant f(x) - f^*.
\end{equation}

We consider gradient-based first-order methods; standard treatments of
this class are given in
\cite{gasnikov2017universal,bubeck2015convex,beck2017first,nemirovski1995information,nesterov2018lectures,polyak1987introduction}.
When gradients are queried at the iterates, the usual linear-span
restriction takes the form
\begin{equation}\label{linear first order method}
    \begin{gathered}
        x^{N + 1} \in x^0 + \operatorname{span} \left\{ \wtgg f(x^0),\wtgg f(x^1),\ldots,\wtgg f(x^{N}) \right\}, \\
        \wtgg f(x^k) \text{ is an approximation for } \nabla f(x^k).
    \end{gathered}
\end{equation}

The exact gradient is assumed to be unavailable. Instead, an inexact
oracle returns $\widetilde{\nabla}f(x)$ satisfying the composite error
condition
\begin{equation}\label{noise condition}
    (\forall x\in\mathbb{R}^n)\qquad \|\wtgg f(x) - \nabla f(x)\|_2 \leqslant \alpha\|\nabla f(x)\|_2 + \delta, \quad 0 \leqslant \alpha < 1, 0 \leqslant \delta.
\end{equation}
Both adaptive procedures use exact function values to evaluate their acceptance conditions. Thus, the inexactness considered here concerns the gradient, not the function values. The proposed methods can also be run using exact specifications of the function and the error, in which case the function values themselves need not be used.

Various models of inexact first-order information have been studied in \cite{devolder2013first,devolder2014first,polyak1987introduction,vasin2023accelerated,gasnikov2017universal,ajalloeian2020convergence}. Our analysis is formulated directly in terms of the norm bound~\eqref{noise condition}. To make its two components explicit, define $\zeta(x) = \wtgg f(x) - \nabla f(x)$. By Lemma~\ref{noise decompose}, condition~\eqref{noise condition} is equivalent to the existence of a decomposition
\begin{equation}\label{noise def abs rel}
    \begin{gathered}
        \zeta(x)=\wtgg f(x)-\nabla f(x) = \zeta_a(x) + \zeta_r(x), \quad \|\zeta_a(x)\|_2 \leqslant \delta, \quad \|\zeta_r(x)\|_2 \leqslant \alpha \|\nabla f(x)\|_2.
    \end{gathered}
\end{equation}
The purely relative model is obtained by setting $\delta=0$:
\begin{equation}\label{eq:relative noise condition}
    (\forall x\in\mathbb{R}^n)\qquad \|\widetilde{\nabla}f(x) - \nabla f(x)\|_2 \leqslant\alpha\|\nabla f(x)\|_2.
\end{equation}
The purely absolute model corresponds to $\alpha=0$:
\begin{equation}\label{eq:absolute noise condition}
    (\forall x\in\mathbb{R}^n) \qquad \|\widetilde{\nabla} f(x) - \nabla f(x)\|_2 \leqslant\delta.
\end{equation}

\noindent \textbf{Contributions.} We develop two adaptive procedures with different information requirements. The first is a gradient descent method that allows adaptation of the smoothness estimate and of both error estimates. Separate switches determine which quantities are updated, while a shared backtracking index controls the active estimations. We obtain convergence for smooth and P\L objectives. Regularization technique yields result for convex problems with purely relative gradient error. The second procedure is the Algorithm~\ref{alg ce-aigm} (CE-AIGM). For strongly convex objectives, it adapts the smoothness estimate while keeping the strong convexity parameter and the error bounds fixed and known. In the fixed-smoothness setting, the accelerated convergence is retained when $\alpha = \mathcal{O}((\mu/L)^{\nicefrac{1}{4}})$, larger admissible relative errors lead to intermediate guarantees (see Theorem~\ref{theorem:ce_aigm_convergence}). 

\noindent \textbf{Organization.} Section~\ref{sec:adaptive_gd} presents adaptive gradient descent and its stationarity, PL, and regularization-based guarantees. Section~\ref{sec:adaptive_inter_method} introduces CE-AIGM and explains its accelerated and intermediate regimes. Auxiliary inequalities, regularization lemmas, and the proofs for the two algorithms are collected in Appendix~\ref{sec:appendix_misc}-\ref{sec:appendix_intermediate}.

\section{Adaptive gradient descent} \label{sec:adaptive_gd}

We first consider gradient descent when the smoothness constant and the parameters of the composite error model are not necessarily known. Algorithm~\ref{alg:adaptive gd} uses a computable backtracking condition to select a step together with the error estimates needed to accept it. It also provides options $F_{\alpha}$, $F_{\delta}$, and $F_{L}$ that enable or disable the corresponding adaptivity. The active parameters nevertheless share a single search index, rather than being estimated through three independent searches. An estimate whose adaptation is disabled must already be a guaranteed upper bound on the corresponding parameter:
\begin{align*}
    F_\alpha=\textbf{false} \Longrightarrow \alpha_0 \geqslant\alpha, \;
    F_\delta=\textbf{false} \Longrightarrow \delta_0 \geqslant\delta, \;
    F_L=\textbf{false} \Longrightarrow L_0 \geqslant L.
\end{align*}
For the maxima defining the search bounds below, we use the conventions
$0/0=1$ and $\log_2 0=-\infty$.

\begin{algorithm}[H]
\caption{Adaptive Gradient Descent $(x^0, N, \alpha_0, F_{\alpha}, \delta_0, F_{\delta}, L_0, F_{L})$}  
\label{alg:adaptive gd}
\begin{algorithmic}[1]
\State
\noindent {\bf Input:} Starting point $x^0$, number of steps $N \geqslant 1$.
\State {\bf Input:} $0 \leqslant \alpha_0 < 1$ - initial estimation of relative error,
\State {\bf Input:} $F_{\alpha} \in \lbrace \textbf{false}, \textbf{true} \rbrace$ - whether to adapt the relative error.
\State {\bf Input:} $\delta_0 \geqslant 0$ - initial estimation of absolute error,
\State {\bf Input:} $F_{\delta} \in \lbrace \textbf{false}, \textbf{true} \rbrace$ - whether to adapt the absolute error.
\State {\bf Input:} $L_0 > 0$ - initial estimation of smoothness,
\State {\bf Input:} $F_L \in \lbrace \textbf{false}, \textbf{true} \rbrace$ -  whether to adapt the smoothness.
\State {\bf Set} $J_0 = 0$.
\For {$k = 0 \dots N - 1$}
    \State $\wha_k, \whdelta_k, \whL_k, \theta_k, x^{k + 1} = \text{Step}\left(J_k, x^k, \alpha_0, F_{\alpha}, \delta_0, F_{\delta}, L_0, F_{L} \right)$.

    \While {$f(x^{k + 1}) > f(x^{k}) - \theta_k \| \wtgg f(x^{k}) \|_2^2 + \frac{1}{4 (1 + \wha_k)^2 \whL_k} \whdelta_k^2$}
        \State $J_k = J_k + 1$.
        \State $\wha_k, \whdelta_k, \whL_k, \theta_k, x^{k + 1} = \text{Step}\left(J_k, x^k, \alpha_0, F_{\alpha}, \delta_0, F_{\delta}, L_0, F_{L} \right)$.
    \EndWhile
   \State $J_{k + 1} = \max \{0, J_{k} - 1 \}$.
\EndFor
\State 
\noindent {\bf Output:} $x^N$.

\State

\Function{Step}{$t, x, \alpha_0, F_{\alpha}, \delta_0, F_{\delta}, L_0, F_{L}$}
    \State $\wha = 1 - \frac{1 - \alpha_0}{2^t}$ \textbf{If} $(F_{\alpha}$ is \textbf{true) else} $\alpha_0$.
    \State $\whdelta = 2^t \delta_0$ \textbf{If} $(F_{\delta}$ is \textbf{true) else} $\delta_0$.
    \State $\whL = 2^t L_0$ \textbf{If} $(F_{L}$ is \textbf{true) else} $L_0$.
    \State $h = \frac{1}{4 \whL} \frac{1 - \wha}{(1 + \wha)^2}$.
    \State $\theta = \frac{1}{32 \whL} \frac{(1 - \wha)^2}{(1 + \wha)^4}$.
    \State $y = x - h \wtgg f(x)$.
    \State
    \Return $\left(\wha, \whdelta, \whL, \theta, y \right)$.
\EndFunction
\end{algorithmic}
\end{algorithm}

The first result controls the smallest true-gradient norm among the first $N$ iterates. It applies without convexity, although the true gradient norm is used only as an analytical measure and is not required by the algorithm.

\begin{theorem} \label{th:nonconvex_gd_adaptive}
    Let the function $f$ be $L$-smooth~\eqref{smooth cond}, $\wtgg f$ satisfies composite error condition~\eqref{noise condition}. Then Algorithm~\ref{alg:adaptive gd} provides convergence:
    \begin{eqnarray*}
        \min_{0 \leqslant k \leqslant N - 1} \|\nabla f(x^k) \|_2^2
        & \leqslant & 1024 \frac{M_*^3 L_0}{(1 - \alpha_0)^2 (1 - \alpha)^2} \frac{f(x^0) - f(x^*)}{N}
        \\
        & + & 512 \frac{M_*^5 \delta_0^2}{(1 - \alpha_0)^2 (1 - \alpha)^2}
    \end{eqnarray*}
    where
    \begin{equation*}
        \begin{gathered}
            J_* = \left\lceil \max \left \lbrace 0, \log_2\left( \frac{1 - \alpha_0}{1 - \alpha} \right), \log_2\left(\frac{L}{L_0} \right), \log_2\left(\frac{\delta}{\delta_0} \right) \right \rbrace \right\rceil,
            \\
            M_* = 2^{J_*}, \text{ assuming } \frac{0}{0} = 1.
        \end{gathered}
    \end{equation*}
    and the total number of inexact gradient $\wtgg f$ calls is bounded by $2N + J_* - 1$.
\end{theorem}

The factor $M_*$ quantifies the scale of the initial underestimation of the parameters. Because the active estimates share a search index, underestimating one parameter may also enlarge the others. The stationarity
statement concerns a best iterate in gradient norm, not necessarily
the returned point $x^N$. Under the PL condition, the same one-step inequality yields a bound
on the last iterate in objective value.

\begin{theorem} \label{th:pl_gd_adaptive}
    Let the function $f$ be $L$-smooth~\eqref{smooth cond} and satisfies P$\L$ condition~\eqref{PL cond}, $\wtgg f$ satisfies composite error condition~\eqref{noise condition}. Then Algorithm~\ref{alg:adaptive gd} ensures the following convergence rate:
    \begin{eqnarray*}    
        f(x^N) - f^* & \leqslant & \left(1 - \frac{1}{512} (1 - \alpha_0)^2 (1 - \alpha)^2 \cdot M_*^{-3} \frac{\mu}{L_0} \right)^N (f(x^0) - f^*)
        \\
        & + & 256 \frac{M_*^5}{(1 - \alpha_0)^{2} (1 - \alpha)^{2}} \frac{\delta_0^2}{\mu},
    \end{eqnarray*} 
    where
    \begin{equation*}
        \begin{gathered}
            J_* = \left\lceil \max \left \lbrace 0, \log_2\left( \frac{1 - \alpha_0}{1 - \alpha} \right), \log_2\left(\frac{L}{L_0} \right), \log_2\left(\frac{\delta}{\delta_0} \right) \right \rbrace \right\rceil,
            \\
            M_* = 2^{J_*}, \text{ assuming } \frac{0}{0} = 1.
        \end{gathered}
    \end{equation*}
    and the total number of inexact gradient $\wtgg f$ calls is bounded by $2N + J_* - 1$.
\end{theorem}

\begin{theorem} \label{th:pl_gd_non_adaptive}
    Let the function $f$ be $L$-smooth~\eqref{smooth cond} and satisfies P$\L$ condition~\eqref{PL cond}, $\wtgg f$ satisfies composite error condition~\eqref{noise condition}. Then Algorithm~\ref{alg:adaptive gd} with $\alpha = \alpha_0, L_0 = L, \delta_0 = \delta, F_{\alpha} = F_{\delta} = F_L = \textbf{false}$ ensures the following convergence rate:
    \begin{eqnarray*}    
        f(x^N) - f^* & \leqslant & \left(1 - \frac{1}{512} (1 - \alpha)^4 \frac{\mu}{L} \right)^N (f(x^0) - f^*)
        + 256 \frac{\delta^2}{(1 - \alpha)^{4}\mu},
    \end{eqnarray*}
\end{theorem}

For a convex objective that need not satisfy the PL condition, we apply
the method to the regularized function $f_\mu(x | x^0)=f(x)+\frac{\mu}{2}\|x-x^0\|_2^2$. Lemma~\ref{lemma:reg noise} explains the correctness of using the Algorithm~\ref{alg:adaptive gd} for such function.

\begin{theorem} \label{th:convex_reg_gd_adaptive}
    Let the function $f$ be $L$-smooth~\eqref{smooth cond} and convex~\eqref{convexity}, $\wtgg f$ satisfies relative error condition~\eqref{eq:relative noise condition}. Consider an arbitrary $0 < \mu, 0 < R_0$, then Algorithm~\ref{alg:adaptive gd} ensures the following convergence rate:
    \begin{eqnarray*}
        f(x^N) - f^* 
        & \leqslant & \left(1 - \frac{1}{512} (1 - \alpha_0)^2 (1 - \alpha)^2 M_*^{-3} \cdot \frac{\mu}{L_0} \right)^N (f(x^0) - f^*)
        \\
        & + &
        256\frac{M_*^5 \mu R_0^2}{(1 - \alpha_0)^{2} (1 - \alpha)^{2}} + \frac{\mu R^2}{2},
    \end{eqnarray*}
    where
    \begin{equation*}
        \begin{gathered}
            J_* = \left\lceil \max \left \lbrace 0, \log_2\left( \frac{1 - \alpha_0}{1 - \alpha} \right), \log_2\left(\frac{L + \mu}{L_0} \right), \log_2\left(\alpha \frac{R}{R_0} \right) \right \rbrace \right\rceil,
            M_* = 2^{J_*}.
        \end{gathered}
    \end{equation*}
    and the total number of inexact gradient $\wtgg f$ calls is bounded by $2N + J_* - 1$. And one can solve problem $f(x^N) - f^* \leqslant \varepsilon$, where $\varepsilon > 0$, choosing:
    \begin{eqnarray*}
        \mu        
        & = & \frac{\varepsilon}{\widehat{R}^2},
    \end{eqnarray*}
    then after:
    \begin{eqnarray*}
        N
        & = & \max \left\{ 0, \left\lceil \frac{512 M_{**}^3}{(1 - \alpha_0)^2(1 - \alpha)^2} \frac{L_0 \widehat{R}^2}{\varepsilon} \ln \left( \frac{3 (f(x^0) - f^*)}{\varepsilon} \right) \right\rceil \right \}
    \end{eqnarray*}
    iterations method will provide $f(x^N) - f^* \leqslant \varepsilon$, where
    \begin{eqnarray*}
        \widehat{R}^2 & = & \max \left\{ \frac{3R^2}{2}, \frac{768 M_{**}^5 R_0^2}{(1- \alpha_0)^2(1 - \alpha)^2} \right\}, \\
        J_{**} & = & \left\lceil \max \left \lbrace 0, \log_2\left( \frac{1 - \alpha_0}{1 - \alpha} \right), \log_2\left(\frac{2 L}{L_0} \right), \log_2\left(\alpha \frac{R}{R_0} \right) \right \rbrace \right\rceil,
        M_{**} = 2^{J_{**}}.
    \end{eqnarray*}
\end{theorem}

\begin{remark} \label{rem:delta_0=0_gd}
    It follows from the Theorem~\ref{th:nonconvex_gd_adaptive}-Theorem~\ref{th:pl_gd_non_adaptive} that, if $\delta_0 = 0$ is chosen, the inner loop may fail to terminate when $\delta > 0$. Therefore, the choice $\delta_0 = 0$ should be made only when it is known that there is no absolute error component in~\eqref{eq:absolute noise condition}.
\end{remark}

The proofs in Appendix~\ref{sec:appendix_adaptive_gd} establish a one-step inequality for every accepted update and then use telescoping or the PL condition, respectively. The finite backtracking bound follows from the existence of a search level at which all active estimates are conservative.
\section{Adaptive intermediate method} \label{sec:adaptive_inter_method}

We now turn to strongly convex objectives and seek faster convergence when the relative gradient error is sufficiently small. The Algorithm~\ref{alg ce-aigm} (CE-AIGM) combines extrapolation with a gradient step and a computable acceptance test. Its convergence bound interpolates between accelerated and non-accelerated condition-number dependence as the relative error parameter increases. Unlike Algorithm~\ref{alg:adaptive gd}, CE-AIGM adapts only the smoothness estimation, the parameter of strong-convexity $\mu$ is also known, whereas $L$ need not be.

\begin{algorithm}[H]
    \caption{ CE-AIGM (Composite Error Adaptive Intermediate Gradient Method).} 
    \label{alg ce-aigm}
    \begin{algorithmic}[1]
        \State
        \noindent {\bf Input:} Starting point $x^0$, number of steps $N \geqslant 1$.
        \State {\bf Input:} $\mu$ - strongly-convexity constant of function $f$.
        \State {\bf Input:} $\alpha$ - relative error of gradient,
        \State {\bf Input:} $\delta$ - absolute error of gradient,
        \State {\bf Input:} $L_0 \geqslant \mu$ - initial estimation of smoothness,
        \State {\bf Set} $J_0 = 0, c_0 = f(x^0), u^0 = x^0, A_0 = 0$.
        \For { $k = 0 \dots N - 1$}
            \State $y^k, u^{k + 1}, x^{k + 1}, c_{k + 1}, A_{k + 1} = \text{Step}(J_k, u^k, x^k, c_k, A_k, \mu, \alpha, \delta, L_0)$.
            \While {$f(x^{k + 1}) > c_{k + 1} + A_{k + 1}$}
                \State $J_k = J_k + 1$.
                \State $y^k, u^{k + 1}, x^{k + 1}, c_{k + 1}, A_{k + 1} = \text{Step}(J_k, u^k, x^k, c_k, A_k, \mu, \alpha, \delta, L_0)$.
            \EndWhile
           \State $J_{k + 1} = \max \{0, J_{k} - 1 \}$.
        \EndFor
        \State
        \noindent {\bf Output:} $x^N$.

    \State

    \Function{Step}{$t, u, x, c, A, \mu, \alpha, \delta, L_0$}
        \State $\whL = 2^t L_0$
        \State $\displaystyle h = \frac{1 - \alpha}{(1 + \alpha)^2} \frac{1}{4 \whL}$,
        \State $s = 1 - 2\alpha + 3\alpha^2, \quad m = 1 - 2\alpha - \alpha^2, \quad q = \displaystyle \frac{(1 - \alpha)^4}{(1 + \alpha)^4} \frac{\mu}{32 \whL}$,
        \State $\omega = \displaystyle \frac{(m - s) + \sqrt{(s - m)^2 +4mq }}{2 m}$,
        \State $\displaystyle y = \frac{\omega u + x}{1 + \omega}, \quad \widetilde{g} = \wtgg f(y), \quad \widehat{x} = y - h \widetilde{g}, \quad \widehat{u} = (1 - \omega) u + \omega y - \frac{2 \omega}{\mu} \widetilde{g}$,
        \State $\widehat{A} = (1 - \omega) A + \displaystyle \frac{\delta^2}{4 (1 + \alpha)^2 \whL} + \displaystyle \omega \frac{4\delta^2}{(1 - \alpha)^2 \mu}$,
        \State \begin{eqnarray*}
        \widehat{c} & = & (1 - \omega) c + \omega \left( f(y) - \frac{s}{(1 - \alpha)^2 \mu} \|\wtgg f(y) \|_2^2 - \frac{2\delta^2}{(1 - \alpha)^2 \mu}  \right) \\
        & + & \frac{\omega (1 - \omega) \mu}{4} \left \|y - u - \frac{2}{\mu} \widetilde{\nabla} f(y) \right \|_2^2,
        \end{eqnarray*}
        \State
        \Return $\left(y, \widehat{u}, \widehat{x}, \widehat{c}, \widehat{A} \right)$.
    \EndFunction
\end{algorithmic}
\end{algorithm}

\begin{theorem} \label{theorem:ce_aigm_convergence}
Let $f$ be an $L$-smooth and $\mu$-strongly convex function, $\widetilde{\nabla} f$ satisfies composite error condition~\eqref{noise condition}, $\alpha \in [0; 1/3]$. Then Algorithm~\ref{alg ce-aigm} with $L_0 \geqslant \mu$ generates $x^N$, s.t. 
\begin{eqnarray*}
    f(x^N) - f^*
    & \leqslant & \exp \left( - \frac{1}{250} \left( \frac{\mu}{L_*} \right)^{1 - 2\gamma^*} N \right) \left( f(x^0) - f^* + \frac{\mu}{4} R^2 \right)
    + 57 \frac{\delta^2}{\mu}.
\end{eqnarray*}
where:
\begin{equation*}
    \begin{gathered}
        R := \|x^0 - x^* \|_2, \gamma^* = \min \left \lbrace \frac{1}{4}, \log_{\frac{\mu}{L_*}} \left(3 \alpha \right) \right \rbrace, \text{ if } \mu = L_* \text{ or } \alpha = 0, \text{ then } \gamma^* = \frac{1}{4}.
        \\
        L_* = \max\limits_{0 \leqslant k \leqslant N - 1} \whL_k \leqslant \max \{ L_0, 2 L \},
        \\
        J_* = \left\lceil \max \left \lbrace 0, \log_2\left(\frac{L}{L_0} \right) \right \rbrace \right\rceil.
    \end{gathered}
\end{equation*}
The total number of iterations of the inner loop can be estimated as $N + J_* - 1$ a 
and the total number of inexact gradient $\wtgg f$ evaluations over $N$ iterations of Algorithm~\ref{alg ce-aigm} is estimated as $2N + J_* - 1$.
\end{theorem}

\begin{theorem} \label{theorem:ce_aigm_convergence_non_adaptive}
Let $f$ be an $L$-smooth and $\mu$-strongly convex function, $\widetilde{\nabla} f$ satisfies composite error condition~\eqref{noise condition}, $\alpha \in [0; 1/3]$. Then Algorithm~\ref{alg ce-aigm} with $L_0 = L$ generates $x^N$, s.t. 
\begin{eqnarray*}
    f(x^N) - f^*
    & \leqslant & \exp \left( - \frac{1}{250} \left( \frac{\mu}{L} \right)^{1 - 2\gamma^*} N \right) \left( f(x^0) - f^* + \frac{\mu}{4} R^2 \right)
    + 57 \frac{\delta^2}{\mu}.
\end{eqnarray*}
where:
\begin{equation*}
    \begin{gathered}
        R := \|x^0 - x^* \|_2, \gamma^* = \min \left \lbrace \frac{1}{4}, \log_{\frac{\mu}{L}} \left(3 \alpha \right) \right \rbrace, \text{ if } \mu = L \text{ or } \alpha = 0, \text{ then } \gamma^* = \frac{1}{4}.
    \end{gathered}
\end{equation*}
\end{theorem}

The dependence on the relative error can be written
without logarithms. Set $\kappa_*=L_*/\mu\geqslant1$. The exponent in
Theorem~\ref{theorem:ce_aigm_convergence} satisfies the identity
\begin{equation*}    
    \left(\frac{\mu}{L_*}\right)^{1-2\gamma^*}
    =\frac{1}{\max\{\sqrt{\kappa_*},\,9\alpha^2\kappa_*\}},
\end{equation*}
hence the accelerated condition-number dependence is retained when $\alpha\leqslant\frac13\kappa_*^{-1/4}$. Above this scale, the second term in the maximum determines the stated bound. At $\alpha=1/3$, it equals $\kappa_*$ and gives the non-accelerated condition-number dependence. Between these regimes, the estimate reflects the size of the relative error. For example, in the fixed-smoothness setting, with $\kappa=L/\mu$ and $\alpha=\frac13\kappa^{-\beta}$ for $0\leqslant\beta\leqslant1/4$, the factor governing the iteration count is $\kappa^{1-2\beta}$. The absolute error enters the objective bound through $\mathcal{O} \left(\delta^2/\mu \right)$, uniformly over the admissible range of $\alpha$, there is no additional amplification of this residual by the condition number, in comparison with the results obtained in~\cite{vasin2023accelerated}.

Analogously to Theorem~\ref{th:convex_reg_gd_adaptive}, a solution to the convex problem using a regularization technique can be proposed.

\begin{theorem} \label{theorem:convex_reg_aigm_adaptive}
    Let the function $f$ be $L$-smooth~\eqref{smooth cond} and convex~\eqref{convexity}, $\wtgg f$ satisfies relative error condition~\eqref{eq:relative noise condition} with $\alpha \in [0;1/3]$ and constant $R = \|x^0 - x^* \|_2$ is known. Consider an arbitrary $0 < \mu \leqslant L_0$, then Algorithm~\ref{alg ce-aigm} ensures the following convergence rate:
    \begin{eqnarray*}
        f(x^N) - f^*
        & \leqslant & \exp \left(-\frac{1}{250} \left( \frac{\mu}{L_*} \right)^{1 - 2\gamma^*} N \right) \left( f(x^0) - f^* + \frac{\mu}{4} R^2 \right)
        + \frac{39 \mu R^2}{2},
    \end{eqnarray*}
    where
    \begin{equation*}
        \begin{gathered}
            L_* = \max\limits_{0 \leqslant k \leqslant N - 1} \whL_k \leqslant \max \{ L_0, 2 (L + \mu) \},
            \\
            \gamma^* = \min \left \lbrace \frac{1}{4}, \log_{\frac{\mu}{L_*}} \left(3 \alpha \right) \right \rbrace, \text{ if } \alpha = 0 \text{ or } \mu = L_*, \text{ then } \gamma^* = \frac{1}{4}.
            \\
            J_* = \left\lceil \max \left \lbrace 0, \log_2\left(\frac{L + \mu}{L_0} \right) \right \rbrace \right\rceil.
        \end{gathered}
    \end{equation*}
    and the total number of inexact gradient $\wtgg f$ calls is bounded by $2N + J_* - 1$. And one can solve problem $f(x^N) - f^* \leqslant \varepsilon$, where $0 < \varepsilon \leqslant f(x^0) - f^*$, choosing (with $L_0 \geqslant \mu$):
    \begin{eqnarray*}
        \mu
        & = & \frac{\varepsilon}{39 R^2} \Longrightarrow \gamma^* = \min \left \lbrace \frac{1}{4}, \log_{\frac{\varepsilon}{39 L_* R^2}} \left(3 \alpha \right) \right \rbrace, \text{ if } \alpha = 0, \text{ then } \gamma^* = \frac{1}{4},
    \end{eqnarray*}
    then after:
    \begin{eqnarray*}
        N
        & = & \max \left\{ 0, \left\lceil 250 \left(\frac{39 L_* R^2}{\varepsilon} \right)^{1 - 2\gamma^*} \ln \left( \frac{157 (f(x^0) - f^*)}{78 \varepsilon} \right) \right\rceil \right \}.
    \end{eqnarray*}
    iterations method will provide $f(x^N) - f^* \leqslant \varepsilon$.
\end{theorem}

\begin{remark} \label{rem:long_tails}
    Using the \textit{LongTails} framework~\cite{longtails2026}, one can obtain analogous results and even improve upon them. Specifically, in the strongly convex case, this framework yields the convergence (for other method) stated in Theorem~\ref{theorem:ce_aigm_convergence_non_adaptive}: for $\alpha < 1$ and $\varepsilon \geqslant 136 \frac{\delta^2}{(1 - \alpha)^2 \mu}$, it ensures that $f(x^N) - f^* \leqslant \varepsilon$ after
    \begin{equation*}
        N =
        \left\lceil
        \left( 8 \sqrt{\frac{L}{\mu (1 - \alpha)}} + \frac{32 \alpha^2 L}{\mu (1 - \alpha)} \right) \log \left(1 + \frac{2 L R^2}{\varepsilon} \right)
        \right\rceil,
     \end{equation*}
    iterations. In the convex case ($\mu = 0$), this framework outperforms Theorem~\ref{theorem:convex_reg_aigm_adaptive} by reducing the logarithmic factor through a damped linear-coupling method (without regularization). For relative error~\eqref{eq:relative noise condition} ($\alpha < 1$), this method guarantees $f(x^N) - f^* \leqslant \varepsilon$ after
    \begin{equation*}
        N = \left\lceil
        7 \sqrt{\frac{LR^2}{(1 - \alpha) \varepsilon}} + \frac{11 \alpha^2 LR^2}{(1 - \alpha) \varepsilon}
        \right\rceil,
    \end{equation*}
    iterations. We invite readers to apply this framework to their own problem formulations and to verify for themselves that their convergence results can be improved.
\end{remark}

\section*{AI Use Statement}

The authors used GPT-6 to check the proofs and to improve the English text (spelling and grammar). All ideas and proofs were conceived without the assistance of artificial intelligence. No AI tools were used to generate mathematical content, ideas, or proofs. For information on the possibility of obtaining results using AI, see Remark~\ref{rem:long_tails}.

\bibliographystyle{sn/tfs}
\bibliography{literature}

\begin{appendix}

\section{Miscellaneous Used Lemmas} \label{sec:appendix_misc}

\begin{lemma} \label{fenchel}
    If $u, v \in \mathbb{R}^n$, then $\forall \lambda > 0$:
    \begin{equation*}
        \langle u, v \rangle \leqslant \frac{\lambda}{2} \|u \|_2^2 + \frac{1}{2 \lambda} \|v \|_2^2.
    \end{equation*}
\end{lemma}

\begin{lemma} \label{(a - b)2 lb}
    If $a, b \geqslant 0$, then
    \begin{equation*}
        (a - b)^2 \geqslant \frac{1}{2} a^2 - b^2.
    \end{equation*}
\end{lemma}
\begin{proof}
    Using Fenchel~\ref{fenchel} inequality with $\lambda = \frac{1}{2}$.
\end{proof}

\begin{lemma} \label{basic alpha conditions}
    If $\widetilde{\nabla} f$ satisfies error condition~\eqref{noise condition}. Following inequalities take place:
    \begin{equation*}
        \begin{gathered}
            (1 - \alpha) \|\nabla f(x) \|_2 - \delta \leqslant \|\widetilde{\nabla} f(x) \|_2 \leqslant (1 + \alpha) \|\nabla f(x) \|_2 + \delta, \\
            \frac{1}{1 + \alpha} \left( \| \widetilde{\nabla} f(x) \|_2 - \delta \right) \leqslant \| \nabla f(x) \|_2 \leqslant \frac{1}{1 - \alpha} \left( \| \widetilde{\nabla} f(x) \|_2 + \delta \right), \\
            \frac{1}{2} (1 - \alpha)^2 \|\nabla f(x) \|_2^2 - \delta^2 \leqslant \|\wtgg f(x) \|_2^2 \leqslant 2 (1 + \alpha)^2 \|\nabla f(x) \|_2^2 + 2 \delta^2, \\
            \frac{1}{2(1 + \alpha)^2} \| \widetilde{\nabla} f(x) \|_2^2 - \frac{1}{(1 + \alpha)^2} \delta^2 \leqslant \|\nabla f(x) \|_2^2 \leqslant \frac{2}{(1 - \alpha)^2} \|\wtgg f(x) \|_2^2 + \frac{2}{(1 - \alpha)^2} \delta^2.
        \end{gathered}
    \end{equation*}
\end{lemma}
\begin{proof}
    Follows from definition~\eqref{noise condition}, Lemma~\ref{fenchel} and Lemma~\ref{(a - b)2 lb}.
\end{proof}

\begin{lemma} \label{lem:composite error inexact estimation}
    If $\widetilde{\nabla} f$ satisfies error condition~\eqref{noise condition}, then:
    \begin{equation*}
        \| \wtgg f(x) - \nabla f(x) \|_2 \leqslant \frac{\alpha}{1 - \alpha} \|\wtgg f(x) \|_2 + \frac{\delta}{1 - \alpha}.
    \end{equation*}
\end{lemma}
\begin{proof}
    \begin{eqnarray*}
        \| \wtgg f(x) - \nabla f(x) \|_2 
        & \leqslant & \alpha \| \nabla f(x) \|_2 + \delta
        \overset{\text{Lemma}~\ref{basic alpha conditions}}{\leqslant} \frac{\alpha}{1 - \alpha} \left( \| \wtgg f(x) \|_2 + \delta \right) + \delta
        \\
        & \leqslant & \frac{\alpha}{1 - \alpha} \|\wtgg f(x) \|_2 + \frac{\delta}{1 - \alpha}.
    \end{eqnarray*}
\end{proof}

\begin{lemma} \label{noise decompose}
    If $\zeta(x) = \widetilde{\nabla} f(x) - \nabla f(x)$ satisfies error condition~\eqref{noise condition}, then exist two vector fields $\zeta_a, \zeta_r$:
    \begin{equation*}
        \zeta(x) = \zeta_a(x) + \zeta_r(x),
    \end{equation*}
    which satisfy:
    \begin{equation*}
        \| \zeta_a(x) \| \leqslant \delta, \hspace{5mm} \| \zeta_r(x) \|_2 \leqslant \alpha \| \nabla f(x) \|_2.
    \end{equation*}
\end{lemma}
\begin{proof}
    Let us define two sets:
    \begin{equation*}
        \begin{gathered}
            S_1 = \lbrace \nabla f(x) + v \; , \; \| v \|_2 \leqslant \alpha \| \nabla f(x) \|_2 \rbrace, \\
            S_2 = \lbrace \widetilde{\nabla} f(x) + v \; , \; \| v \|_2 \leqslant \delta \rbrace.
        \end{gathered}
    \end{equation*}
    Then from condition $\| \widetilde{\nabla} f(x) - \nabla f(x) \|_2 \leqslant \alpha \| \nabla f(x) \|_2 + \delta$ we can draw a conclusion about $S_1 \cap S_2 \not= \emptyset $ and the existence of two vectors $\zeta_r, \zeta_a$, described in theorem condition.
\end{proof}

\begin{lemma} \label{lem:dot_product_lb_inexact}
    If $\wtgg f$ satisfies composite error~\eqref{noise condition}, then:
    \begin{eqnarray*}
        (\forall x \in \mathbb{R}^n) \quad \langle \nabla f(x), \wtgg f(x) \rangle
        & \geqslant & \frac{1 - \alpha}{4 (1 + \alpha)^2} \| \wtgg f(x) \|_2^2  - \frac{\delta^2}{1 - \alpha}.
    \end{eqnarray*}
\end{lemma}
\begin{proof}
    \begin{eqnarray*}
        \langle \nabla f(x), \wtgg f(x) \rangle
        & \overset{\text{Lemma}~\ref{noise decompose}}{=} & \langle \nabla f(x) + \zeta_a(x) + \zeta_r(x), \nabla f(x) \rangle
        \\
        & \overset{\text{Lemma}~\ref{fenchel}, \eqref{noise def abs rel}}{\geqslant} & \| \nabla f(x) \|_2^2 - \alpha \| \nabla f(x) \|_2^2 - \frac{\delta^2}{2 \lambda} - \frac{\lambda}{2} \|\nabla f(x) \|_2^2
        \\
        & \overset{\lambda = 1 - \alpha}{\geqslant} & \frac{1 - \alpha}{2} \|\nabla f(x) \|_2^2 - \frac{\delta^2}{2 (1 - \alpha)}
        \\
        & \overset{\text{Lemma}~\ref{basic alpha conditions}}{\geqslant} & \frac{1 - \alpha}{2} \left( \frac{1}{2 (1 + \alpha)^2} \| \wtgg f(x) \|_2^2 - \frac{\delta^2}{(1 + \alpha)^2} \right) 
        \\
        & - & \frac{\delta^2}{2 (1 - \alpha)}
        \\
        & = & \frac{1 - \alpha}{4 (1 + \alpha)^2} \| \wtgg f(x) \|_2^2 - \frac{1 - \alpha}{2 (1 + \alpha)^2} \delta^2 - \frac{1}{2 (1 - \alpha)} \delta^2
        \\
        & \overset{\nicefrac{1 - \alpha}{1 + \alpha} \leqslant 1}{\geqslant} & \frac{1 - \alpha}{4 (1 + \alpha)^2} \| \wtgg f(x) \|_2^2  - \frac{\delta^2}{1 - \alpha}.
    \end{eqnarray*}
\end{proof}

\begin{lemma} \label{pl dist}
    Let $f$ satisfies strong convexity~\eqref{strong conv}, then:
    \begin{equation*}
        \|x - x^* \|_2 \leqslant \frac{1}{\mu} \| \nabla f(x) \|_2.
    \end{equation*}
\end{lemma}
\begin{proof}
    From strong convexity and gradient dominance condition (which follows from the strong convexity):
    \begin{equation*}
        \begin{gathered}    
            \frac{\mu}{2} \|x - x^* \|_2^2 \leqslant f(x) - f^* \leqslant \frac{1}{2\mu} \|\nabla f(x) \|_2^2, \quad \|x - x^* \|_2 \leqslant \frac{1}{\mu} \|\nabla f(x) \|_2.
        \end{gathered}
    \end{equation*}
\end{proof}

\begin{lemma} \label{lem:gd_inexact_step_inexact_bound}
    Let $f$ be $L$-smooth~\eqref{smooth cond}, $\widetilde{\nabla} f$ satisfies composite error condition~\eqref{noise condition}. If variables $x, y$ are related by the following relation $x = y - h \wtgg f(y)$, where:
    \begin{equation} \label{gd const def}
        \begin{gathered}
            h = \frac{1 - \alpha}{(1 + \alpha)^2} \frac{1}{4 L},
        \end{gathered}
    \end{equation}
    then:
    \begin{equation*}
        f(x) \leqslant f(y) - \frac{(1 - \alpha)^2}{(1 + \alpha)^4} \frac{1}{32 L} \|\wtgg f(y) \|_2^2 + \frac{\delta^2}{4 (1 + \alpha)^2 L}.
    \end{equation*}
\end{lemma}
\begin{proof}
Using smoothness of function $f$ we get:
\begin{equation*}
    f(x) \leqslant f(y) + \langle \nabla f(y), x - y \rangle + \frac{L}{2} \| x - y \|_2^2.
\end{equation*}
Let us estimate dot product:
\begin{eqnarray*}
    \langle \nabla f(y), x - y \rangle 
    & = & -h \langle \nabla f(y), \wtgg f(y) \rangle
    \overset{\text{Lemma}~\ref{lem:dot_product_lb_inexact}}{\leqslant} 
    -h \bigg( \frac{1 - \alpha}{4 (1 + \alpha)^2} \| \wtgg f(y) \|_2^2  - \frac{\delta^2}{1 - \alpha} \bigg).
\end{eqnarray*}
Then:
\begin{eqnarray}
    f(x) 
    & \leqslant & f(y) \nonumber
    - h \left( \frac{1 - \alpha}{4 (1 + \alpha)^2} \| \wtgg f(y) \|_2^2  - \frac{\delta^2}{1 - \alpha} \right) + \frac{L h^2}{2} \|\wtgg f(y) \|_2^2 \\ \nonumber
    & = & f(y) + h \|\wtgg f(y) \|_2^2 \left(\frac{L h}{2} - \frac{1 - \alpha}{4 (1 + \alpha)^2} \right) +h \frac{\delta^2}{1 - \alpha} \nonumber \\
    & \overset{h \text{ definition}}{=} & f(y) - \frac{1 - \alpha}{(1 + \alpha)^2} \frac{1}{4 L} \frac{1 - \alpha}{8(1 + \alpha)^2} \|\wtgg f(y) \|_2^2 + \frac{\delta^2}{4 (1 + \alpha)^2 L} \\ \label{h gd step ineq}
    & = & f(y) - \frac{(1 - \alpha)^2}{(1 + \alpha)^4} \frac{1}{32 L} \|\wtgg f(y) \|_2^2 + \frac{\delta^2}{4 (1 + \alpha)^2 L}
\end{eqnarray}
\end{proof}

\section{Regularization} \label{sec:appendix_regularization}

Let us consider convex~\eqref{convexity} function $f$.

\begin{equation} \label{reg def}
    \begin{gathered}
        f_{\mu}(x | x^0) = f(x) + \frac{\mu}{2} \|x - x^0 \|_2^2,  \mu > 0, \\
        x_\mu^* \text{ - unique solution for problem: } f_{\mu}(\cdot | x^0) \to \min\limits_{x \in \mathbb{R}^n}, \\
        f_{\mu}^* = f_{\mu}(x_{\mu}^* | x^0), \quad R_\mu = \|x_{\mu}^* - x^0 \|_2, \\
        L_{\mu} = L + \mu, \quad
        \wtgg f_{\mu}(x) = \wtgg f(x) + \mu (x - x^0).
    \end{gathered}
\end{equation}

\begin{lemma} \label{lemma:reg dist bound}
Let $f_{\mu}$ be a regularization function for $f$. Then\; $R_{\mu} \leqslant R.$
\end{lemma}
\begin{proof}
We have
\begin{equation*}
    f^* + \frac{\mu}{2} R_{\mu}^2 \leqslant f(x_{\mu}^*) + \frac{\mu}{2} R_{\mu}^2 = f_{\mu}^* \leqslant f_{\mu}(x^* | x^0) = f(x^*) + \frac{\mu}{2} R^2. 
\end{equation*}
So, $R_{\mu}^2 \leqslant R^2.$
\end{proof}

\begin{lemma} \label{lemma:reg func bound}
Let $f_{\mu}$ be a regularization function for $f$. Then\; $\forall x \in \mathbb{R}^n$:
\begin{equation*}
    f(x) - f^* \leqslant f_{\mu}(x | x^0) - f^*_{\mu} + \frac{\mu}{2} R^2.
\end{equation*}
\end{lemma}
\begin{proof}
We have
\begin{eqnarray*}
    f(x) - f^*
    & \leqslant & f(x) + \frac{\mu}{2} \|x - x^0 \|_2^2 - f^* 
    = f_{\mu}(x | x^0) - f^*
    \\
    & = & f_{\mu}(x | x^0) - f_{\mu}(x^* | x^0) + \frac{\mu}{2} R^2 \leqslant f_{\mu}(x | x^0) - f_{\mu}^* + \frac{\mu}{2} R^2. 
\end{eqnarray*}
\end{proof}

\begin{lemma} \label{lemma:reg noise}
Let $f$ be a $L$-smooth~\eqref{smooth cond}, convex~\eqref{convexity} function, and $\wtgg f$ satisfies relative error condition~\eqref{eq:relative noise condition}. Then the gradient estimation of regularized function $f_{\mu}$ satisfies~\eqref{noise condition}, that is
\begin{equation*}
    \| \wtgg f_{\mu}(x) - \nabla f_{\mu}(x) \|_2 \leqslant \alpha \|\nabla f_{\mu}(x) \|_2 + \alpha \mu R.
\end{equation*}
\end{lemma}
\begin{proof}
Since $f_{\mu}$ is strongly convex:
\begin{equation*}
    \langle \nabla f_{\mu}(x), x - x_{\mu}^* \rangle \geqslant \mu \|x - x_{\mu}^* \|_2^2.
\end{equation*}
Then:
\begin{eqnarray*}
    \| \nabla f_{\mu}(x) - \mu (x - x_{\mu}^*) \|_2^2
    & = & \| \nabla f_{\mu}(x) \|_2^2 - 2 \mu \langle \nabla f_{\mu}(x), x - x_{\mu}^* \rangle + \mu^2 \| x - x_{\mu}^* \|_2^2
    \\
    & \leqslant & \| \nabla f_{\mu}(x) \|_2^2 - \mu^2 \| x - x_{\mu}^* \|_2^2 \leqslant  \| \nabla f_{\mu}(x) \|_2^2,
\end{eqnarray*}
That is $\| \nabla f_{\mu}(x) - \mu (x - x_{\mu}^*) \|_2 \leqslant \| \nabla f_{\mu}(x) \|_2$. Thus we can provide error bound:
\begin{eqnarray*}
    \| \wtgg f_{\mu}(x) - \nabla f_{\mu}(x) \|_2
    & = & \|\wtgg f(x) + \mu (x - x^0) - \nabla f(x) - \mu(x - x^0) \|_2
    \\
    & \leqslant & \alpha \|\nabla f(x) \|_2 = \alpha \|\nabla f(x) + \mu (x - x^0) - \mu(x - x^0) \|_2
    \\
    & = & \alpha \|\nabla f_{\mu}(x) - \mu(x - x^*_{\mu}) + \mu(x_0 - x^*_{\mu}) \|_2
    \\
    & \leqslant & \alpha \|\nabla f_{\mu}(x) \|_2 + \alpha \mu R_{\mu}
    \overset{\text{Lemma}~\ref{lemma:reg dist bound}}{\leqslant}  \alpha \|\nabla f_{\mu}(x) \|_2 + \alpha \mu R.
\end{eqnarray*}
\end{proof}

\section{Adaptivity missing proofs} \label{sec:appendix_adaptive_gd}

\begin{lemma} \label{lemma:adaptive_gradient_step}
    Let $f$ is $L$-smooth~\eqref{smooth cond}, $\widetilde{\nabla} f$ satisfies composite error condition~\eqref{noise condition}. If $1 > \widehat{\alpha} \geqslant \alpha, \widehat{L} \geqslant L, \widehat{\delta} \geqslant \delta$ and $x = y - h \wtgg f(y)$, where:
    \begin{equation*}
        \begin{gathered}
            h = \frac{1 - \wha}{(1 + \wha)^2} \frac{1}{4 \whL},
        \end{gathered}
    \end{equation*}
    then:
    \begin{equation*}
        f(x) \leqslant f(y) -  \frac{(1 - \wha)^2}{(1 + \wha)^4}  \frac{1}{32 \whL} \|\wtgg f(y) \|_2^2
    + \frac{\whdelta ^2}{4 (1 + \wha)^2 \whL}.
    \end{equation*}
\end{lemma}
\begin{proof}
     Observe that if a function $f$ satisfies the smoothness condition with parameter $L$, then it also satisfies this condition with any $\whL \geqslant L$, similarly, it satisfies the corresponding conditions for $\wha \geqslant \alpha$ and $\whdelta \geqslant \delta$. Consequently, Lemma~\ref{lem:gd_inexact_step_inexact_bound} can be applied.
\end{proof}

\begin{lemma} \label{lem:adaptive_gradient_gap_estimation}
    Let $f$ be $L$-smooth~\eqref{smooth cond}, and let $\widetilde{\nabla} f$ satisfy the composite error condition~\eqref{noise condition}. Then, for Algorithm~\ref{alg:adaptive gd} and for all $k = 0,\dots,N - 1$,
    \begin{equation*}
        f(x^{k + 1}) \leqslant f(x^k) - 2^{-3 J_* - 10} \frac{\left(1 - \alpha_0 \right)^2 (1 - \alpha)^2}{L_0} \|\nabla f(x^k) \|_2^2 + 2^{2J_* - 1} \frac{\delta_0^2}{L_0},
    \end{equation*}
    where (assuming $\frac{0}{0} = 1$)
    \begin{equation*}
        \max\limits_{0 \leqslant k \leqslant N - 1} J_k \leqslant J_* = \left \lceil \max \left \lbrace 0, \log_2\left( \frac{1 - \alpha_0}{1 - \alpha} \right), \log_2\left(\frac{L}{L_0} \right), \log_2\left(\frac{\delta}{\delta_0} \right) \right \rbrace \right \rceil,
    \end{equation*}
    The total number of inexact gradient $\wtgg f$ evaluations over $N$ iterations of Algorithm~\ref{alg:adaptive gd} is estimated as $2N + J_* - 1$.
\end{lemma}
\begin{proof}
    At each iteration of the outer loop of Algorithm~\ref{alg:adaptive gd}, a "while" loop is started and terminates when
    \begin{equation} \label{eq:in_adapt_gap_proof}
         f(x^{k + 1}) \leqslant f(x^k) -  \frac{(1 - \wha_k)^2}{(1 + \wha_k)^4} \frac{1}{32 \whL_k} \|\wtgg f(x^k) \|_2^2 + \frac{\whdelta_k ^2}{4\whL_k (1 + \wha_k)^2}.
    \end{equation}
    From the definitions $\alpha_0 \leqslant \wha_k \leqslant 1 - \frac{1 - \alpha_0}{2^{J_k}}$, $\delta_0 \leqslant \whdelta_k \leqslant 2^{J_k} \delta_0$, and $L_0 \leqslant \whL_k \leqslant 2^{J_k} L_0$, and from Lemma~\ref{lemma:adaptive_gradient_step}, it follows that the "while" loop is guaranteed to terminate once $\wha_k \geqslant \alpha$, $\whdelta_k \geqslant \delta$, and $\whL_k \geqslant L$, which is equivalent to $J_k \geqslant \max \left \lbrace 0, \log_2\left( \frac{1 - \alpha_0}{1 - \alpha} \right), \log_2\left(\frac{L}{L_0} \right), \log_2\left(\frac{\delta}{\delta_0} \right) \right \rbrace$. Hence $J_k \leqslant J_* = \left\lceil\max \left \lbrace 0, \log_2\left( \frac{1 - \alpha_0}{1 - \alpha} \right), \log_2\left(\frac{L}{L_0} \right), \log_2\left(\frac{\delta}{\delta_0} \right) \right \rbrace \right\rceil$. Then, substituting the estimations (and using $1 \leqslant 1 + \wha_k \leqslant 2$) of the adaptive parameters into \eqref{eq:in_adapt_gap_proof}, we obtain
    \begin{eqnarray*}
        f(x^{k + 1}) 
        & \leqslant & f(x^k) - 2^{-3 J_* - 9} \frac{\left(1 - \alpha_0 \right)^2}{L_0} \|\wtgg f(x^k) \|_2^2 + 2^{2J_* - 2} \frac{\delta_0^2}{L_0}
        \\
        & \overset{\text{Lemma}~\ref{basic alpha conditions}}{\leqslant} & f(x^k) - 2^{-3 J_* - 9} \frac{\left(1 - \alpha_0 \right)^2}{L_0} \left( \frac{1}{2} (1 - \alpha)^2 \|\nabla f(x^k) \|_2^2 - \delta^2 \right)
        \\
        & + & 2^{2J_* - 2} \frac{\delta_0^2}{L_0}
        \\
        & \leqslant & f(x^k) - 2^{-3 J_* - 10} \frac{\left(1 - \alpha_0 \right)^2 (1 - \alpha)^2}{L_0} \|\nabla f(x^k) \|_2^2
        \\
        & + &  2^{-3 J_* - 9} \frac{\left(1 - \alpha_0 \right)^2}{L_0} \delta^2 + 2^{2J_* - 2} \frac{\delta_0^2}{L_0}
        \\
        & \overset{\delta \leqslant 2^{J_*} \delta_0, 0 \leqslant \alpha_0}{\leqslant} & f(x^k) - 2^{-3 J_* - 10} \frac{\left(1 - \alpha_0 \right)^2 (1 - \alpha)^2}{L_0} \|\nabla f(x^k) \|_2^2
        + 2^{2J_* - 1} \frac{\delta_0^2}{L_0}.
    \end{eqnarray*}
    Now we estimate the total number of iterations, taking into account the "while" loop steps. On the $k$-th iteration of the outer loop, the inner loop takes at most $J_{k + 1} - J_k + 1$ steps, hence the total number is
    \begin{equation*}
        J_0 + \sum_{k = 0}^{N - 2} (J_{k + 1} - J_k + 1) \leqslant N - 1 + J_{N - 1} \leqslant N - 1 + J_*.
    \end{equation*}
    Then the total number of inexact gradient calls estimated as $2 N + J_* - 1$.
\end{proof}

\begin{proof}[Proof of the Theorem~\ref{th:nonconvex_gd_adaptive}]
    Using Lemma~\ref{lem:adaptive_gradient_gap_estimation} one can obtain inequality for $k =0 \dots N - 1$:
    \begin{equation*}
        2^{-3 J_* -10} \frac{\left(1 - \alpha_0 \right)^2 (1 - \alpha)^2}{L_0} \|\nabla f(x^k) \|_2^2 \leqslant f(x^k) - f(x^{k + 1}) + 2^{2J_* - 1} \frac{\delta_0^2}{L_0}.
    \end{equation*}
    Summing up inequalities above for $k = 0$ to $k = N - 1$:
    \begin{eqnarray*}
        N \min_{0 \leqslant k \leqslant N - 1} \|\nabla f(x^k) \|_2^2
        & \leqslant &
        \sum_{k = 0}^{N - 1} \|\nabla f(x^k) \|_2^2 
        \\
        & \leqslant & \frac{2^{3 J_* + 10} L_0}{\left(1 - \alpha_0 \right)^2 (1 - \alpha)^2} \left( f(x^0) - f(x^N) + N 2^{2J_* - 1} \frac{\delta_0^2}{L_0} \right).
    \end{eqnarray*}
    That is:
    \begin{eqnarray*}
        \min_{0 \leqslant k \leqslant N - 1} \|\nabla f(x^k) \|_2^2
        & \leqslant & 2^{3 J_* + 10} \frac{L_0}{(1 - \alpha_0)^2 (1 - \alpha)^2} \frac{f(x^0) - f(x^*)}{N}
        \\
        & + & 2^{5J_* + 9} \frac{\delta_0^2}{(1 - \alpha_0)^2 (1 - \alpha)^2}.
        \\
        & = & 1024 \frac{M_*^3 L_0}{(1 - \alpha_0)^2 (1 - \alpha)^2} \frac{f(x^0) - f(x^*)}{N}
        \\
        & + & 512 \frac{M_*^5 \delta_0^2}{(1 - \alpha_0)^2 (1 - \alpha)^2}.
    \end{eqnarray*}
    Total amount of inexact gradient calculations estimated as $2N + J_* - 1$ by Lemma~\ref{lem:adaptive_gradient_gap_estimation}.
\end{proof}

\begin{proof}[Proof of the Theorem~\ref{th:pl_gd_adaptive}]
    Using Lemma~\ref{lem:adaptive_gradient_gap_estimation} one can obtain inequality for $k =0 \dots N - 1$:
    \begin{equation*}
        f(x^{k + 1}) - f^* \leqslant f(x^k) - f^* - 2^{-3 J_* - 10} \frac{\left(1 - \alpha_0 \right)^2 (1 - \alpha)^2}{L_0} \|\nabla f(x^k) \|_2^2 + 2^{2J_* - 1} \frac{\delta_0^2}{L_0}.
    \end{equation*}
    Then, using P$\L$ condition~\eqref{PL cond}:
    \begin{eqnarray*}
        f(x^{k + 1}) - f^*
        & \leqslant &  f(x^k) - f^* - 2^{-3 J_* - 10} \frac{\left(1 - \alpha_0 \right)^2 (1 - \alpha)^2}{L_0} \|\nabla f(x^k) \|_2^2 + 2^{2J_* - 1} \frac{\delta_0^2}{L_0}
        \\
        & \leqslant & f(x^k) - f^* - 2^{-3 J_* - 9} \frac{\left(1 - \alpha_0 \right)^2 (1 - \alpha)^2 \mu}{L_0} (f(x^k) - f^*)  + 2^{2J_* - 1} \frac{\delta_0^2}{L_0}
        \\
        & = & \left(1 - (1 - \alpha_0)^2 (1 - \alpha)^2 \cdot 2^{-3 J_* - 9} \frac{\mu}{L_0} \right)(f(x^k) - f^*)  + 2^{2J_* - 1} \frac{\delta_0^2}{L_0}.
    \end{eqnarray*}
    Iterating inequality above, denoting $q = 1 - (1 - \alpha_0)^2 (1 - \alpha)^2 \cdot 2^{-3 J_* - 9} \frac{\mu}{L_0}$, since $L_0 2^{J_*} \geqslant L \geqslant \mu$, then:
    \begin{eqnarray*}
        f(x^{N}) - f^*
        & \leqslant & q^N (f(x^0) - f^*) + \sum_{k = 0}^{N - 1} q^k 2^{2J_* - 1} \frac{\delta_0^2}{L_0}
        \\
        & \leqslant & \left(1 - (1 - \alpha_0)^2 (1 - \alpha)^2 \cdot 2^{-3 J_* - 9} \frac{\mu}{L_0} \right)^N (f(x^0) - f^*)
        \\
        & + & (1 - \alpha_0)^{-2} (1 - \alpha)^{-2} \cdot 2^{3 J_* + 9} \frac{L_0}{\mu} 2^{2J_* - 1} \frac{\delta_0^2}{L_0}
        \\
        & \leqslant & \left(1 - (1 - \alpha_0)^2 (1 - \alpha)^2 \cdot 2^{-3 J_* - 9} \frac{\mu}{L_0} \right)^N (f(x^0) - f^*)
        \\
        & + & \frac{2^{5 J_* + 8}}{(1 - \alpha_0)^{2} (1 - \alpha)^{2}} \frac{\delta_0^2}{\mu}
        \\
        & = & 
        \left(1 - \frac{1}{512} (1 - \alpha_0)^2 (1 - \alpha)^2 \cdot M_*^{-3} \frac{\mu}{L_0} \right)^N (f(x^0) - f^*)
        \\
        & + & 256 \frac{M_*^5}{(1 - \alpha_0)^{2} (1 - \alpha)^{2}} \frac{\delta_0^2}{\mu}.
    \end{eqnarray*}
    The total number of inexact gradient evaluations is estimated as $2N + J_* - 1$ by Lemma~\ref{lem:adaptive_gradient_gap_estimation}.
\end{proof}

\begin{proof}[Proof of the Theorem~\ref{th:pl_gd_non_adaptive}]
    Result follows directly from Theorem~\ref{th:pl_gd_adaptive}, since $J_* = 0$ and $M_* = 1$.
\end{proof}

\begin{proof}[Proof of the Theorem~\ref{th:convex_reg_gd_adaptive}]
    Let us consider $f_{\mu} = f(x) + \frac{\mu}{2} \| x - x^0 \|_2^2$. It us known, that $f_{\mu}$ is $\mu$ strongly convex~\eqref{strong conv} and $L + \mu$ smooth~\eqref{smooth cond}. Corresponding gradient estimation $\wtgg f_{\mu} = \wtgg f + \mu (x - x^0)$ satisfies composite error condition with relative error parameter $\alpha$ and absolute error $\alpha \mu R$ (by Lemma~\ref{lemma:reg noise}). Let us use $\delta_0 = \mu R_0$ as absolute error initial estimation. Using Theorem~\ref{th:pl_gd_adaptive} convergence for $f_{\mu}$ one can obtain:
    \begin{eqnarray*}    
        f_{\mu}(x^N | x^0) - f_{\mu}^* & \leqslant & \left(1 - (1 - \alpha_0)^2 (1 - \alpha)^2 \cdot 2^{-3 J_* - 9} \frac{\mu}{L_0} \right)^N (f(x^0) - f^*)
        \\
        & + & \frac{2^{5 J_* + 8} \mu R_0^2}{(1 - \alpha_0)^{2} (1 - \alpha)^{2}},
    \end{eqnarray*}
    where:
    \begin{equation*}
        \begin{gathered}
            J_* = \left\lceil \max \left \lbrace 0, \log_2\left( \frac{1 - \alpha_0}{1 - \alpha} \right), \log_2\left(\frac{L + \mu}{L_0} \right), \log_2\left(\alpha \frac{R}{R_0} \right) \right \rbrace \right\rceil,
            M_* = 2^{J_*}.
        \end{gathered}
    \end{equation*}
    Then (since $f_{\mu}^* = f(x_{\mu}^*) + \frac{\mu}{2} \|x_{\mu}^* - x^0 \|_2^2 \geqslant f^*$):
    \begin{eqnarray*}
        f(x^N) - f^* 
        & \overset{\text{Lemma}~\ref{lemma:reg func bound}}{\leqslant} & f_{\mu}(x^N | x^0) - f_{\mu}^* + \frac{\mu}{2} R^2
        \\
        & \leqslant & \left(1 - (1 - \alpha_0)^2 (1 - \alpha)^2 \cdot 2^{-3 J_* - 9} \frac{\mu}{L_0} \right)^N (f(x^0) - f^*)
        \\
        & + & \frac{2^{5 J_* + 8} \mu R_0^2}{(1 - \alpha_0)^{2} (1 - \alpha)^{2}} + \frac{\mu R^2}{2}
        \\
        & = & \left(1 - \frac{1}{512} (1 - \alpha_0)^2 (1 - \alpha)^2 M_*^{-3} \cdot \frac{\mu}{L_0} \right)^N (f(x^0) - f^*)
        \\
        & + &
        256\frac{M_*^5 \mu R_0^2}{(1 - \alpha_0)^{2} (1 - \alpha)^{2}} + \frac{\mu R^2}{2}
        \\
        & \leqslant & \exp \left(-\frac{1}{512} (1 - \alpha_0)^2 (1 - \alpha)^2 M_*^{-3} \cdot \frac{\mu}{L_0} N \right) (f(x^0) - f^*)
        \\
        & + &
        256\frac{M_*^5 \mu R_0^2}{(1 - \alpha_0)^{2} (1 - \alpha)^{2}} + \frac{\mu R^2}{2}.
    \end{eqnarray*}
    Thus, to solve problem $f(x^N) - f^* \leqslant \varepsilon$, where $0 < \varepsilon \leqslant f(x^0) - f^*$ one can choose:
    \begin{eqnarray*}
        \mu
        & = & \varepsilon \cdot \min \left\{ \frac{2}{3R^2}, \frac{(1 - \alpha_0)^2 (1 - \alpha)^2}{768 M_{**}^5 R_0^2} \right\}.
    \end{eqnarray*}
    Such chose of $\mu$ guarantee:
    \begin{equation*}
        \mu \leqslant \frac{2 \varepsilon}{3R^2} \leqslant \frac{LR^2}{3 R^2} \leqslant L,
    \end{equation*}
    where:
    \begin{equation*}
        J_{**} = \left\lceil \max \left \lbrace 0, \log_2\left( \frac{1 - \alpha_0}{1 - \alpha} \right), \log_2\left(\frac{2 L}{L_0} \right), \log_2\left(\alpha \frac{R}{R_0} \right) \right \rbrace \right\rceil \geqslant J_*,
        M_{**} = 2^{J_{**}}.
    \end{equation*}
    then after:
    \begin{eqnarray*}
        N
        & = & \max \left\{ 0, \left\lceil \frac{512 M_{**}^3}{(1 - \alpha_0)^2(1 - \alpha)^2} \frac{L_0}{\mu} \ln \left( \frac{3 (f(x^0) - f^*)}{\varepsilon} \right) \right\rceil \right \}
        \\
        & = & \max \left\{ 0, \left\lceil \frac{512 M_{**}^3}{(1 - \alpha_0)^2(1 - \alpha)^2} \frac{L_0 \widehat{R}^2}{\varepsilon} \ln \left( \frac{3 (f(x^0) - f^*)}{\varepsilon} \right) \right\rceil \right \}
    \end{eqnarray*}
    iterations method will provide $f(x^N) - f^* \leqslant \varepsilon$, where
    \begin{equation*}
        \widehat{R}^2 = \max \left\{ \frac{3R^2}{2}, \frac{768 M_{**}^5 R_0^2}{(1- \alpha_0)^2(1 - \alpha)^2} \right\}.
    \end{equation*}
\end{proof}

\section{Adaptive intermediate method missing proofs} \label{sec:appendix_intermediate}

Following \citep[p. 83]{nesterov2018lectures}, let us introduce a parameterized set of functions $\Psi$, its element defined for $c \in \mathbb{R}, \kappa \in \mathbb{R}^{++},$ and $u \in \mathbb{R}^n$,  as follows
\begin{equation} \label{qudaratic family}
    \psi(x | c, \kappa, u) = c + \frac{\kappa}{2} \| x - u \|_2^2, \quad \forall x \in \mathbb{R}^n. 
\end{equation}

According to~\cite{nesterov2018lectures}, we can mention to the following useful property of the class $\Psi$.

\begin{lemma} \label{psi comb}
Let $\psi_1, \psi_2 \in \Psi$. Then $\forall c_1, c_2 \in \mathbb{R}, \forall \eta_1, \eta_2, \kappa_1, \kappa_2 \in \mathbb{R}^{++}$, and $\forall u, v \in \mathbb{R}^n$, we have 
\[
    \eta_1 \psi_1(x|c_1, \kappa_1, u) + \eta_2 \psi_2(x|c_2, \kappa_2, v) = \psi_3(x|c_3, \kappa_3, w), 
\]
where 
\begin{equation*}
    \begin{gathered}
        c_3 = \eta_1 c_1 + \eta_2 c_2 + \frac{\eta_1\eta_2\kappa_1\kappa_2}{2(\eta_1\kappa_1 + \eta_2\kappa_2)} \|u  - v \|_2^2, \\
        \kappa_3 = \eta_1\kappa_1 + \eta_2\kappa_2, \quad w = \frac{\eta_1\kappa_1 u +\eta_2\kappa_2 v}{\eta_1\kappa_1 + \eta_2\kappa_2}.
    \end{gathered}
\end{equation*}
\end{lemma}

Also from~\cite{nesterov2018lectures} we mention the following simple lemma.
\begin{lemma} \label{conv lemma}
Let $0 < \lambda < 1$, $A \geqslant 0$, $\psi_0, \psi \in \Psi$, $z \in \mathbb{R}^n$ such that 
\begin{equation*}
f(z) \leqslant \min\limits_{x \in \mathbb{R}^n} \psi(x) + A, \quad \psi(x) \leqslant \lambda \psi_0(x) + (1 - \lambda) f(x), \quad \forall x \in \mathbb{R}^n.
\end{equation*}
Then
\begin{equation*}
    f(z) - f^* \leqslant \lambda (\psi_0(x^*) - f^*) + A. 
\end{equation*}
\end{lemma}
\begin{proof}
    \begin{eqnarray*}
        f(z) - f^*
        & \leqslant & \min\limits_{x \in \mathbb{R}^n} \psi(x) + A - f^*
        \\
        & \leqslant & \psi(x^*) - f^* + A \leqslant \lambda \psi_0(x^*) + (1 - \lambda) f(x^*) - f^* + A
        \\
        & = & \lambda ( \psi_0(x^*) - f^* ) + A. 
    \end{eqnarray*}
\end{proof}

\begin{lemma} \label{accelerated evaluated quadratic estimation}
    Let $\wtgg f$ satisfies error condition~\eqref{noise condition} and $\mu > 0$, then:
    \begin{eqnarray*}
        (\forall x, z \in \R^n) \; \frac{\mu}{4} \left\|x - z + \frac{2}{\mu} \widetilde{\nabla} f(z) \right \|_2^2
        & \leqslant & \langle \nabla f(z), x - z \rangle + \frac{\mu}{2} \|x - z \|_2^2
        \\
        & + & \frac{1 - 2\alpha + 3\alpha^2}{(1 - \alpha)^2 \mu} \|\wtgg f(z) \|_2^2 + \frac{2\delta^2}{(1 - \alpha)^2 \mu}, \\
        (\forall x, z \in \R^n) \; \frac{\mu}{4} \left\|x - z + \frac{2}{\mu} \widetilde{\nabla} f(z) \right \|_2^2
        & \geqslant & \langle \nabla f(z), x - z \rangle
        + \frac{1 - 2\alpha - \alpha^2}{(1 - \alpha)^2 \mu} \|\wtgg f(z) \|_2^2
        \\
        & - & \frac{2\delta^2}{(1 - \alpha)^2 \mu}.
    \end{eqnarray*}
\end{lemma}
\begin{proof}
We will find bounds for each term of sum:
\begin{equation*}
    \frac{\mu}{4} \left\|x - z + \frac{2}{\mu} \widetilde{\nabla} f(z) \right \|_2^2 = \frac{\mu}{4} \|x - z \|_2^2 + \langle \wtgg f(z), x - z \rangle + \frac{1}{\mu} \|\wtgg f(z) \|_2^2.
\end{equation*}
Linear form:
\begin{eqnarray*}
    \langle \wtgg f(z), x - z \rangle
    & \overset{\text{Lemma }\ref{fenchel}}{\leqslant} & \langle \nabla f(z), x - z \rangle + \frac{1}{\mu} \left\| \wtgg f(z) - \nabla f(z)  \right \|_2^2 + \frac{\mu}{4} \|x - z \|_2^2 \\
    & \overset{\text{Lemma}~\ref{lem:composite error inexact estimation}}{\leqslant} &  \langle \nabla f(z), x - z \rangle + \frac{1}{\mu} \left( \frac{\alpha}{1 - \alpha} \|\wtgg f(z) \|_2 + \frac{\delta}{1 - \alpha} \right)^2
    \\
    & + & \frac{\mu}{4} \|x - z \|_2^2
    \\
    & \leqslant &  \langle \nabla f(z), x - z \rangle + \frac{2\alpha^2}{(1 - \alpha)^2 \mu} \|\wtgg f(z) \|_2^2
    \\
    & + & \frac{2\delta^2}{(1 - \alpha)^2 \mu} + \frac{\mu}{4} \|x - z \|_2^2.
\end{eqnarray*}
\begin{eqnarray*}
    \langle \wtgg f(z), x - z \rangle
    & \overset{\text{Lemma }\ref{fenchel}}{\geqslant} & \langle \nabla f(z), x - z \rangle - \frac{1}{\mu} \| \wtgg f(z) - \nabla f(z) \|_2^2 - \frac{\mu}{4} \|x - z \|_2^2
    \\
    & \overset{\text{Lemma}~\ref{lem:composite error inexact estimation}}{\geqslant} &  \langle \nabla f(z), x - z \rangle -  \frac{1}{\mu} \left( \frac{\alpha}{1 - \alpha} \|\wtgg f(z) \|_2 + \frac{\delta}{1 - \alpha} \right)^2
    \\
    & - & \frac{\mu}{4} \|x - z \|_2^2
    \\
    & \geqslant &  \langle \nabla f(z), x - z \rangle - \frac{2\alpha^2}{(1 - \alpha)^2 \mu} \|\wtgg f(z) \|_2^2
    \\
    & - & \frac{2\delta^2}{(1 - \alpha)^2 \mu} - \frac{\mu}{4} \|x - z \|_2^2.
\end{eqnarray*}
Then:
\begin{eqnarray*}
    \frac{\mu}{4} \left\|x - z + \frac{2}{\mu} \widetilde{\nabla} f(z) \right \|_2^2
    & \leqslant & \langle \nabla f(z), x - z \rangle + \frac{\mu}{2} \|x - z \|_2^2
    \\
    & + & \frac{2\alpha^2}{(1 - \alpha)^2 \mu} \|\wtgg f(z) \|_2^2 + \frac{2\delta^2}{(1 - \alpha)^2 \mu} + \frac{1}{\mu} \|\wtgg f(z) \|_2^2
    \\
    & = & 
    \langle \nabla f(z), x - z \rangle + \frac{\mu}{2} \|x - z \|_2^2
    \\
    & + & \frac{1 - 2\alpha + 3\alpha^2}{(1 - \alpha)^2 \mu} \|\wtgg f(z) \|_2^2 + \frac{2\delta^2}{(1 - \alpha)^2 \mu}.
\end{eqnarray*}
\begin{eqnarray*}
    \frac{\mu}{4} \left\|x - z + \frac{2}{\mu} \widetilde{\nabla} f(z) \right \|_2^2
    & \geqslant & \langle \nabla f(z), x - z \rangle
    \\
    & - & \frac{2\alpha^2}{(1 - \alpha)^2 \mu} \|\wtgg f(z) \|_2^2 - \frac{2\delta^2}{(1 - \alpha)^2 \mu} + \frac{1}{\mu} \|\wtgg f(z) \|_2^2
    \\
    & = & 
    \langle \nabla f(z), x - z \rangle
    + \frac{1 - 2\alpha - \alpha^2}{(1 - \alpha)^2 \mu} \|\wtgg f(z) \|_2^2 - \frac{2\delta^2}{(1 - \alpha)^2 \mu}.
\end{eqnarray*}

\end{proof}

Now, let us proof the following lemma which provides a lower bound of the function $f$ from the set $\Psi$. 

\begin{lemma} \label{lower bound}
Let $f$ be a $\mu$-strongly convex~\eqref{strong conv} function and $\widetilde{\nabla}{f}$ satisfies composite error condition~\eqref{noise condition}. Then for any $z \in \mathbb{R}^n$, there is $\phi_{z} \in \Psi$, such that: 
\begin{equation*}
    \begin{gathered}
        \phi_{z}(x) = c_{\phi} + \frac{\kappa_{\phi}}{2} \|x - u_{\phi} \|_2^2 \leqslant f(x), \quad \forall x \in \mathbb{R}^n. \\
        c_{\phi} := f(z) - \frac{1 - 2\alpha + 3\alpha^2}{(1 - \alpha)^2 \mu} \|\wtgg f(z) \|_2^2 - \frac{2\delta^2}{(1 - \alpha)^2 \mu}, \\
        \kappa_{\phi} := \mu / 2, \\
        u_{\phi} := z - \frac{2}{\mu} \widetilde{\nabla}f(z).
    \end{gathered}
\end{equation*}
\end{lemma}
\begin{proof}
Proof follows directly from Lemma~\ref{accelerated evaluated quadratic estimation} and strong convexity definition~\eqref{strong conv}.
\end{proof}

\begin{lemma} \label{lem:w solution estimation}
Let $\alpha \in [0, 1/3], 0 < \mu \leqslant L$ and denote the following variables:
\begin{equation*}
    \begin{gathered}
        s = 1 - 2\alpha + 3\alpha^2, \quad m = 1 - 2\alpha - \alpha^2, \quad q = \frac{(1 - \alpha)^4}{(1 + \alpha)^4} \frac{\mu}{32 L} \\
        \omega \text{ is the largest root of the quadratic equation } m \omega^2 + (s - m) \omega - q = 0.
    \end{gathered}
\end{equation*}
Then:
\begin{equation*}
    0 < \frac{1}{250} \left(\frac{\mu}{L}\right)^{1 - 2\gamma^*} \leqslant \omega \leqslant \frac{1}{2} \left(\frac{\mu}{L}\right)^{1 - 2\gamma^*} < 1,
\end{equation*}
where $\gamma^* = \min \left\{ \frac{1}{4}, \log_{\nicefrac{\mu}{L}}(3 \alpha) \right\}$ (if $\mu = L$ or $\alpha = 0$, then $\gamma^* = \frac{1}{4}$).
\end{lemma}
\begin{proof}
We can proof, that $0 < m \leqslant 1$:

\begin{eqnarray*}
    m & = & 1 - 2\alpha - \alpha^2 \overset{\alpha \leqslant \nicefrac{1}{3}}{\geqslant} 1 - \frac{2}{3} - \frac{1}{9} = \frac{2}{9} > 0, \\
    m & = & 1 - 2\alpha - \alpha^2 \overset{\alpha \geqslant 0}{\leqslant} 1.
\end{eqnarray*}


\noindent From the definition of $\gamma^*$ we can make an estimation $\alpha \leqslant \frac{1}{3} \left( \frac{\mu}{L} \right)^{\gamma^*}$.
We can provide upper bound for linear term coefficient:
\begin{eqnarray*}
    s - m 
    & = & 1 - 2\alpha + 3\alpha^2 - 1 + 2\alpha + \alpha^2 = 4 \alpha^2 \leqslant \frac{4}{9} \left( \frac{\mu}{L} \right)^{2\gamma^*}.
\end{eqnarray*}

From $q$ definition:
\begin{equation*}
    \frac{1}{512} \cdot \frac{\mu}{L} \overset{\alpha \leqslant 1 / 3}{\leqslant} q = \frac{(1 - \alpha)^4}{(1 + \alpha)^4} \frac{\mu}{32 L}
    \overset{\alpha \geqslant 0}{\leqslant} \quad \frac{1}{32} \cdot \frac{\mu}{L}.
\end{equation*}

Let us prove by contradiction, that $\omega \geqslant \frac{1}{250} \left(\frac{\mu}{L}\right)^{1 - 2\gamma^*}$. To that end, we assume $\omega < \frac{1}{250} \left(\frac{\mu}{L}\right)^{1 - 2\gamma^*}$. Then:

\begin{align*}
    m \omega^2 + (s - m) \omega - q 
    & < \frac{m}{250^2} \left(\frac{\mu}{L}\right)^{2(1 - 2\gamma^*)} + \frac{s - m}{250}\left(\frac{\mu}{L}\right)^{1 - 2\gamma^*} - q \\
    & \overset{m \leqslant 1}{\leqslant} \frac{1}{250^2} \left(\frac{\mu}{L}\right)^{2(1 - 2\gamma^*)} + \frac{2}{1125} \frac{\mu}{L} - \frac{1}{512} \frac{\mu}{L} \\ 
    & \leqslant \frac{\mu}{L} \left( \frac{1}{62500} \left(\frac{\mu}{L}\right)^{1 - 4\gamma^*} + \frac{2}{1125} - \frac{1}{512} \right) \overset{\gamma^* \leqslant \nicefrac{1}{4}}{\leqslant} - \frac{11473}{72000000}\frac{\mu}{L} < 0.
\end{align*}
Thus, we come to a contradiction because $\omega$ is a root of the equation $m \omega^2 + (s - m) \omega - q = 0$.

Let us also prove by contradiction, that $\omega \leqslant \frac{1}{2} \left(\frac{\mu}{L}\right)^{1 - 2\gamma^*}$. So, we assume $\omega > \frac{1}{2} \left(\frac{\mu}{L}\right)^{1 - 2\gamma^*}$. Then, there are two cases:

\begin{enumerate}
    \item $\gamma^* = \frac{1}{4}$:
    \begin{eqnarray*}
        m \omega^2 + (s - m) \omega - q
        = m \omega^2 + 4\alpha^2 \omega - q
        > \frac{2}{36} \frac{\mu}{L} - \frac{1}{32} \frac{\mu}{L} > 0.
    \end{eqnarray*}
    \item $\gamma^* < \frac{1}{4} \Longrightarrow \alpha = \frac{1}{3} \left(\frac{\mu}{L}\right)^{\gamma^*}$:
    \begin{eqnarray*}
        m \omega^2 + (s - m) \omega - q
        & > & \frac{2}{36} \left(\frac{\mu}{L}\right)^{2 - 4\gamma^*} + \frac{2}{9} \frac{\mu}{L} - \frac{1}{32} \frac{\mu}{L}
        > 0.
    \end{eqnarray*}
\end{enumerate}

In both cases we come to a contradiction, that is $\omega \leqslant \frac{1}{2} \left(\frac{\mu}{L}\right)^{1 - 2\gamma^*}$.

\end{proof}

\begin{lemma} \label{lem:accelerated_step_condition}
    Let $f$ be a $\mu$-strongly convex~\eqref{strong conv} and $L$-smooth function, $\widetilde{\nabla}{f}$ satisfies composite error condition~\eqref{noise condition} with $\alpha \leqslant \frac{1}{3}$. Let $x, u \in \R^n, A \geqslant 0, c \in \R$, such that:
    \begin{equation*}
        \begin{gathered}
            \psi(v) = c + \frac{\mu}{4} \|v - u \|_2^2, \quad f(x) \leqslant c + A,
            \\
            h = \frac{1 - \alpha}{(1 + \alpha)^2} \frac{1}{4 L}, \quad q = \frac{(1 - \alpha)^4}{(1 + \alpha)^4} \frac{\mu}{32 L},
            \\ 
            \omega \text{ -- positive root of } \omega (1 - 2\alpha + 3 \alpha^2) - \omega (1 - \omega) (1 - 2\alpha - \alpha^2) = q,
            \\
            y = \frac{\omega u + x}{1 + \omega}, \quad \widehat{x} = y - h \wtgg f(y),
            \\
            \widehat{\psi}(v) = (1 - \omega) \psi(v) + \omega \phi_{y}(v), 
        \end{gathered}
    \end{equation*}
    where $\phi_{y}(v)$ -- function defined at Lemma~\ref{lower bound}, that is:
    \begin{equation*}
        \phi_{y}(v) = f(y) - \frac{1 - 2\alpha + 3\alpha^2}{(1 - \alpha)^2 \mu} \|\wtgg f(y) \|_2^2 - \frac{2\delta^2}{(1 - \alpha)^2 \mu} + \frac{\mu}{4} \left\|v - y + \frac{2}{\mu} \wtgg f(y) \right \|_2^2.
    \end{equation*}
    Then:
    \begin{equation*}
        f(\widehat{x}) \leqslant \min_{v \in \R^n} \widehat{\psi}(v) + (1 - \omega) A
    + \frac{\delta^2}{4 (1 + \alpha)^2 L} + \omega \frac{4\delta^2}{(1 - \alpha)^2 \mu},
    \end{equation*}
    where:
    \begin{eqnarray*}
        \min_{v \in \R^n} \widehat{\psi}(v) & = & (1 - \omega) c + \omega \left( f(y) - \frac{1 - 2\alpha + 3\alpha^2}{(1 - \alpha)^2 \mu} \|\wtgg f(y) \|_2^2 - \frac{2\delta^2}{(1 - \alpha)^2 \mu}  \right)
        \\
        & + & \frac{\omega (1 - \omega) \mu}{4} \left \|y - u - \frac{2}{\mu} \widetilde{\nabla} f(y) \right \|_2^2.
    \end{eqnarray*}
\end{lemma}
\begin{proof}
Lemma~\ref{lem:w solution estimation} guarantees, that $0 < \omega < 1$. Using Lemma~\ref{psi comb}:
\begin{eqnarray*}
    \widehat{c} 
    & = & \min_{v \in \R^n} \widehat{\psi}(v) = (1 - \omega) c + \omega \left( f(y) - \frac{1 - 2\alpha + 3\alpha^2}{(1 - \alpha)^2 \mu} \|\wtgg f(y) \|_2^2 - \frac{2\delta^2}{(1 - \alpha)^2 \mu}  \right)
    \\
    & + & \frac{\omega (1 - \omega) \mu}{4} \left \|y - u - \frac{2}{\mu} \widetilde{\nabla} f(y) \right \|_2^2.
\end{eqnarray*}
From Lemma~\ref{accelerated evaluated quadratic estimation} we can estimate last term:
\begin{eqnarray*}
    \frac{\omega (1 - \omega) \mu}{4} \left \|y - u - \frac{2}{\mu} \widetilde{\nabla} f(y) \right \|_2^2 
    & \geqslant & \omega (1 - \omega) \langle \nabla f(y), u - y \rangle
    \\
    & + & \omega (1 - \omega) \frac{1 - 2\alpha - \alpha^2}{(1 - \alpha)^2 \mu} \|\wtgg f(y) \|_2^2 - \omega(1 - \omega) \frac{2\delta^2}{(1 - \alpha)^2 \mu}.
\end{eqnarray*}
Grouping the corresponding terms:
\begin{eqnarray*}
    \widehat{c}
    & \geqslant & (1 - \omega) c + \omega f(y) + \omega (1 - \omega) \langle \nabla f(y), u - y \rangle
    \\
    & - & \frac{1}{(1 - \alpha)^2 \mu} \left( \omega (1 - 2\alpha + 3 \alpha^2) - \omega (1 - \omega) (1 - 2\alpha - \alpha^2) \right) \|\wtgg f(y) \|_2^2
    \\
    & - & \omega (2 - \omega) \frac{2\delta^2}{(1 - \alpha)^2 \mu}.
\end{eqnarray*}
From convexity of function $f$:
\begin{eqnarray*}
    (1 - \omega) c
    & \geqslant & (1 - \omega) \left( f(x) - A \right)
    \\
    & \geqslant & (1 - \omega) \left( f(y) + \langle \nabla f(y), x - y \rangle \right) - (1 - \omega) A,
\end{eqnarray*}
then, using definition $y$ we can eliminate the linear form:
\begin{equation*}
     \langle \nabla f(y), \omega (1 - \omega) (u - y) + (1 - \omega) (x - y) \rangle = 0
\end{equation*}
and using definition of $\omega$ we continue $\widehat{c}$ estimation:
\begin{eqnarray*}
    \widehat{c}
    & \overset{\omega \geqslant 0}{\geqslant} & f(y) - (1 - \omega) A
    - \frac{q}{(1 - \alpha)^2 \mu} \|\wtgg f(y) \|_2^2
    - \omega \frac{4\delta^2}{(1 - \alpha)^2 \mu}
    \\
    & \geqslant & f(y) - \frac{(1 - \alpha)^2}{(1 + \alpha)^4} \frac{1}{32 L} \|\wtgg f(y) \|_2^2 - (1 - \omega) A - \omega \frac{4\delta^2}{(1 - \alpha)^2 \mu}
\end{eqnarray*}
Since $\widehat{x} = y - h \wtgg f(y)$, where $h$ selected consistently with the Lemma~\ref{lem:gd_inexact_step_inexact_bound} we can apply that Lemma:
\begin{equation*}
        f(\widehat{x}) \leqslant f(y) - \frac{(1 - \alpha)^2}{(1 + \alpha)^4} \frac{1}{32 L} \|\wtgg f(y) \|_2^2 + \frac{\delta^2}{4 (1 + \alpha)^2 L},
\end{equation*}
thus:
\begin{eqnarray*}
    \widehat{c}
    & \geqslant & f(\widehat{x}) - (1 - \omega) A
    - \frac{\delta^2}{4 (1 + \alpha)^2 L} - \omega \frac{4\delta^2}{(1 - \alpha)^2 \mu}.
\end{eqnarray*}
\end{proof}

Now we can proof convergence of Algorithm~\ref{alg ce-aigm}. We assume that the method accepts $L_0 \geqslant \mu$. Let us denote $\whL_k, \omega_k, c_k, A_k$ as values selected at step $k$ after the completion of the "while" loop.  Let us define, recursively, a pair of sequences $\{\psi_k(x)\}_{k \geqslant 0}, \{\lambda_k\}_{k \geqslant 0}$ corresponding to Algorithm~\ref{alg ce-aigm} as follows:
\begin{align}
    &  \lambda_0 = 1, \quad {c_0 = f(x^0)}, \quad \lambda_{k+1} = (1 - \omega_k) \lambda_k, \quad \forall k \geqslant 0, \label{eq:lambda_def}
    \\& \psi_k(x| c_k, \mu / 2, u^k) = c_k + \frac{\mu}{4} \|x - u^k\|_2^2, \quad \forall k \geqslant 0.  \label{psi^k_Alg1}, \\
    & \psi_{k + 1}(x | c_{k + 1}, \mu / 2, u^{k + 1}) = (1 - \omega_k) \psi_k(x | c_k, \mu / 2, u^k) +  \omega_k \phi_{y^k}(x), \quad \forall k \geqslant 0, \label{psi^k_combination} \\
    & \phi_{y^k}(x) = f(y^k) - \frac{1 - 2\alpha + 3\alpha^2}{(1 - \alpha)^2 \mu} \|\wtgg f(y^k) \|_2^2 - \frac{2\delta^2}{(1 - \alpha)^2 \mu} + \frac{\mu}{4} \left\|x - y^k + \frac{2}{\mu} \wtgg f(y^k) \right \|_2^2, \\
    & A_{k + 1} = (1 - \omega_k) A_k + \frac{\delta^2}{4 (1 + \alpha)^2 \whL_k} + \omega_k \frac{4\delta^2}{(1 - \alpha)^2 \mu}, \quad A_0 = 0. \label{eq:Ak_def}
\end{align}

\begin{lemma} \label{lem:alg_ce_aigm_correctnes}
    Let functions $\{ \psi_k \}_k$ defined at~\eqref{psi^k_Alg1} and~\eqref{psi^k_combination}. Then:
    \begin{eqnarray*}
            c_{k + 1} & = & (1 - \omega_k) c_k + \omega_k \left( f(y^k) - \frac{1 - 2\alpha + 3\alpha^2}{(1 - \alpha)^2 \mu} \|\wtgg f(y^k) \|_2^2 - \frac{2\delta^2}{(1 - \alpha)^2 \mu} \right)
            \\
            & + & \frac{\omega_k (1 - \omega_k) \mu}{4} \left \|y^k - u^k - \frac{2}{\mu} \widetilde{\nabla} f(y^k) \right \|_2^2, \\
            u^{k + 1} & = & (1 - \omega_k) u^k + \omega_k y^k - \frac{2 \omega_k}{\mu} \widetilde{\nabla} f(y^k).
    \end{eqnarray*}
    That is sequence $\{ u^k \}_k$ in the Algorithm~\ref{alg ce-aigm} corresponds sequence $\{ \psi_k \}_k$.
\end{lemma}
\begin{proof}
    Since $\psi_{k + 1}(x | c_{k + 1}, \mu / 2, u^{k + 1}) = (1 - \omega_k) \psi_k(x | c_k, \mu / 2, u^k) + \omega_k \phi_{y^k}(x)$ and $\psi_k, \phi_{y^k} \in \Psi$, we can apply Lemma~\ref{psi comb} and obtained the relations described above. 
\end{proof}

\begin{lemma} \label{lem:adaptive_step_condition}
    Let $f$ be a $\mu$-strongly convex~\eqref{strong conv} and $L$-smooth function, $\widetilde{\nabla}{f}$ satisfies composite error condition~\eqref{noise condition} with $\alpha \leqslant \frac{1}{3}$, then $\forall k \geqslant 0 \; f(x^k) \leqslant c_{k} + A_k$ for Algorithm~\ref{alg ce-aigm} with $L_0 \geqslant \mu$. In this case, the total number of iterations of the inner loop can be estimated as $N + J_* - 1$, where:
    \begin{equation*}
        \max\limits_{0 \leqslant k \leqslant N - 1} J_k \leqslant J_* = \left\lceil \max \left \lbrace 0, \log_2\left(\frac{L}{L_0} \right) \right \rbrace \right\rceil.
    \end{equation*}
    And the total number of inexact gradient $\wtgg f$ evaluations over $N$ iterations of Algorithm~\ref{alg ce-aigm} is estimated as $2N + J_* - 1$.
\end{lemma}
\begin{proof}
    Condition $f(x^k) \leqslant c_k + A_k$ follows from definition of $c_k, A_k$ and point $x^k$ obtained after "while" loop termination. It should be noted that $f(x^0) = c_0 + A_0$, then using induction and Lemma~\ref{lem:accelerated_step_condition} from $f(x^k) \leqslant c_k + A_k$ follows $f(x^{k + 1}) \leqslant c_{k + 1} + A_{k + 1}$. By Lemma~\ref{lem:accelerated_step_condition} the inner loop is guaranteed to terminate, when $\whL_k \geqslant L$, thus the total number of inner loop iterations over the entire execution of the algorithm can be estimated as:
    \begin{equation*}
        J_0 + \sum_{k = 0}^{N - 2} (J_{k + 1} - J_k + 1) = J_{N - 1} + N - 1 \leqslant J_* + N - 1,
    \end{equation*}
    and inexact gradient $\wtgg f$ calls as $2 N + J_* - 1$.
\end{proof}

\begin{lemma} \label{upper bound for lemma}
Let $\{\lambda_k\}_{k \geqslant 0}$ be a sequence defined in~\eqref{eq:lambda_def} and $\{\psi_k\}_{k \geqslant 0}$ be the corresponding functions~\eqref{psi^k_Alg1}. Then 
\begin{equation*}
    \forall k \geqslant 0 \;\; \text{and} \;\; \forall x \in \mathbb{R}^n \quad \psi_k(x) \leqslant \lambda_k \psi_0(x) + (1 - \lambda_k) f(x).
\end{equation*}
\end{lemma}
\begin{proof}
By the principle of induction. For $k = 0$, the inequality is obvious. Assume that inequality takes place for $k$. From \eqref{psi^k_combination} and Lemma \ref{lower bound}, we get 
\begin{align*}
    \psi_{k + 1}(x)  &= (1 - \omega_k) \psi_k(x) + \omega_k \phi_{y^k}(x) 
    \\& \leqslant (1 - \omega_k) \left( (1 - \lambda_k) f(x) + \lambda_k \psi_0(x) \right) + \omega_k f(x) \\
    & = \lambda_{k + 1} \psi_0(x) + (1 - \lambda_{k + 1}) f(x).
\end{align*}
\end{proof}

\begin{proof}[Proof of the Theorem~\ref{theorem:ce_aigm_convergence}]
    From Lemma~\ref{lem:adaptive_step_condition} we obtain:
    \begin{equation*}
        f(x^N) \leqslant c_N + A_N = \min_{x \in \R^n} \psi_N(x) + A_N.
    \end{equation*}
    Applying Lemma~\ref{upper bound for lemma} and then Lemma~\ref{conv lemma}:
    \begin{eqnarray*}
        f(x^N) - f^* \leqslant \lambda_N (\psi_0(x^*) - f^*) + A_N
        \overset{\eqref{psi^k_Alg1}}{=} \lambda_N \left( f(x^0) - f^* + \frac{\mu}{4} \|x^0 - x^* \|_2^2 \right) + A_N.
    \end{eqnarray*}

    From definition of $\{ \lambda_k \}_k$~\eqref{eq:lambda_def} we obtain $\lambda_N = \prod\limits_{k = 0}^{N - 1} (1 - \omega_k)$. Let us denote $\gamma_k^* = \min \left\{ \frac{1}{4}, \log_{\frac{\mu}{\whL_k}} \left(3 \alpha \right) \right\}$ (if $\mu = \whL_k$ or $\alpha = 0$, then $\gamma_k^* = \frac{1}{4}$), then using Lemma~\ref{lem:w solution estimation} -- $\frac{1}{250} \left( \frac{\mu}{\whL_k} \right)^{1 - 2\gamma^*_k} \leqslant \omega_k$. Since $\whL_k \leqslant L_*$, then $\gamma^* \leqslant \gamma_k^*$, that is $\frac{1}{250} \left( \frac{\mu}{L_*} \right)^{1 - 2\gamma^*} \leqslant \omega_k$. Then $\lambda_N \leqslant \left(1 - \frac{1}{250} \left( \frac{\mu}{L_*} \right)^{1 - 2\gamma^*} \right)^N \leqslant \exp \left( - \frac{1}{250} \left( \frac{\mu}{L_*} \right)^{1 - 2\gamma^*} N \right)$.

    Let us denote: 
    \begin{equation*}
        b_k = \frac{\delta^2}{4 (1 + \alpha)^2 \whL_k} + \omega_k \frac{4\delta^2}{(1 - \alpha)^2 \mu},
    \end{equation*}
    then $A_N$ can be expressed as (using its definition~\ref{eq:Ak_def}).
    \begin{equation*}
        A_N
        = \sum_{k = 0}^{N - 1} b_k \prod_{j = k + 1}^{N - 1} (1 - \omega_j).
    \end{equation*}
    Since $\omega_k$ is the root of the quadratic equation $m \omega_k^2 + (s - m) \omega_k - q_k = 0$, where:
    \begin{equation*}
        s = 1 - 2\alpha + 3\alpha^2, \quad m = 1 - 2\alpha - \alpha^2, \quad q_k = \frac{(1 - \alpha)^4}{(1 + \alpha)^4} \frac{\mu}{32 \whL_k}.
    \end{equation*}
    From proof of the Lemma~\ref{lem:w solution estimation} it follows, that $m > 0$ and $\omega_k < 1$. Then $q_k = s \omega_k + m \omega_k (\omega_k - 1) \leqslant s \omega_k$. Thus:
    \begin{eqnarray*}
        b_k
        & = & \frac{\delta^2}{4 (1 + \alpha)^2 \whL_k} + \omega_k \frac{4\delta^2}{(1 - \alpha)^2 \mu}
        = \frac{8(1 + \alpha)^2 \delta^2 q_k}{(1 - \alpha)^4 \mu} + \omega_k \frac{4\delta^2}{(1 - \alpha)^2 \mu}
        \\
        & \overset{q_k \leqslant s \omega_k}{\leqslant} & 
        \left(\frac{8(1 + \alpha)^2 s}{(1 - \alpha)^4} + \frac{4}{(1 - \alpha)^2} \right) \omega_k \frac{\delta^2}{\mu}
        \\
        & = & \left(\frac{8(1 + \alpha)^2 (1 - 2\alpha + 3\alpha^2)}{(1 - \alpha)^4} + \frac{4}{(1 - \alpha)^2} \right) \omega_k \frac{\delta^2}{\mu}
        \\
        & = & \left(\frac{8(1 + \alpha)^2}{(1 - \alpha)^2} \left(1 +  \frac{2\alpha^2}{(1 - \alpha)^2} \right) + \frac{4}{(1 - \alpha)^2} \right) \omega_k \frac{\delta^2}{\mu}
        \\
        & \overset{\alpha \leqslant \nicefrac{1}{3}}{\leqslant} & \left(8 \cdot 4 \left(1 +  \frac{1}{2} \right) + 9 \right) \omega_k \frac{\delta^2}{\mu} = 57 \omega_k \frac{\delta^2}{\mu}.
    \end{eqnarray*}
    Hence:
    \begin{eqnarray*}
        A_N 
        & \leqslant & 57 \frac{\delta^2}{\mu} \sum_{k = 0}^{N - 1} \omega_k \prod_{j = k + 1}^{N - 1} (1 - \omega_j) = 57 \frac{\delta^2}{\mu} \left(1 - \prod_{k = 0}^{N - 1} (1 - \omega_k) \right)
        \\
        & = & 57 (1 - \lambda_N) \frac{\delta^2}{\mu} \leqslant 57 \frac{\delta^2}{\mu}.
    \end{eqnarray*}
    Combining inequalities above one can obtain convergence:
    \begin{eqnarray*}
        f(x^N) - f^*
        & \leqslant & \exp \left( - \frac{1}{250} \left( \frac{\mu}{L_*} \right)^{1 - 2\gamma^*} N \right) \left( f(x^0) - f^* + \frac{\mu}{4} R^2 \right)
        + 57 \frac{\delta^2}{\mu}.
    \end{eqnarray*}
\end{proof}

\begin{proof}[Proof of the Theorem~\ref{theorem:ce_aigm_convergence_non_adaptive}]
    Since $L_0 = L$, then $\forall k \geqslant 0, \; \whL_k = L_* = L_0 = L$ (see Theorem~\ref{theorem:ce_aigm_convergence}) and we can apply previous Theorem.
\end{proof}

\begin{proof}[Proof of the Theorem~\ref{theorem:convex_reg_aigm_adaptive}]
    Similar to the proof of Theorem~\ref{th:convex_reg_gd_adaptive}, let us consider $f_{\mu} = f(x) + \frac{\mu}{2} \| x - x^0 \|_2^2$. It us known, that $f_{\mu}$ is $\mu$ strongly convex~\eqref{strong conv} and $L + \mu$ smooth~\eqref{smooth cond}. Corresponding gradient estimation $\wtgg f_{\mu} = \wtgg f + \mu (x - x^0)$ satisfies composite error condition with relative error parameter $\alpha$ and absolute error $\alpha \mu R$ (by Lemma~\ref{lemma:reg noise}). Using Theorem~\ref{theorem:ce_aigm_convergence} convergence for $f_{\mu}$ one can obtain:
    \begin{eqnarray*}
        f_{\mu}(x^N | x^0) - f_{\mu}^*
        & \leqslant & \exp \left( - \frac{1}{250} \left( \frac{\mu}{L_*} \right)^{1 - 2\gamma^*} N \right) \left( f_{\mu}(x^0 | x^0) - f_{\mu}^* + \frac{\mu}{4} R_{\mu}^2 \right)
        + 57 \frac{\delta^2}{\mu},
    \end{eqnarray*}
    where:
    \begin{equation*}
        \begin{gathered}
            L_* = \max\limits_{0 \leqslant k \leqslant N - 1} \whL_k \leqslant \max \{ L_0, 2 (L + \mu) \},
            \\
            \gamma^* = \min \left \lbrace \frac{1}{4}, \log_{\frac{\mu}{L_*}} \left(3 \alpha \right) \right \rbrace, \text{ if } L_* = \mu \text{ or }\alpha = 0, \text{ then } \gamma^* = \frac{1}{4}.
            \\
            J_* = \left\lceil \max \left \lbrace 0, \log_2\left(\frac{L + \mu}{L_0} \right) \right \rbrace \right\rceil.
        \end{gathered}
    \end{equation*}
    Then (since $f_{\mu}^* = f(x_{\mu}^*) + \frac{\mu}{2} \|x_{\mu}^* - x^0 \|_2^2 \geqslant f^*$):
    \begin{eqnarray*}
        f(x^N) - f^* 
        & \overset{\text{Lemma}~\ref{lemma:reg func bound}}{\leqslant} & f_{\mu}(x^N | x^0) - f_{\mu}^* + \frac{\mu}{2} R^2
        \\
        & \overset{\text{Lemma}~\ref{lemma:reg dist bound}}{\leqslant} & \exp \left(-\frac{1}{250} \left( \frac{\mu}{L_*} \right)^{1 - 2\gamma^*} N \right) \left( f(x^0) - f^* + \frac{\mu}{4} R^2 \right)
        \\
        & + & 57 \alpha \mu R^2 + \frac{\mu R^2}{2}
        \\
        & \overset{\alpha \leqslant \nicefrac{1}{3}}{\leqslant} & \exp \left(-\frac{1}{250} \left( \frac{\mu}{L_*} \right)^{1 - 2\gamma^*} N \right) \left( f(x^0) - f^* + \frac{\mu}{4} R^2 \right)
        \\
        & + & \frac{39 \mu R^2}{2}.
    \end{eqnarray*}
    Thus, to solve problem $f(x^N) - f^* \leqslant \varepsilon$, where $0 < \varepsilon \leqslant f(x^0) - f^*$ one can choose (with $L_0 \geqslant \mu$):
    \begin{eqnarray*}
        \mu
        & = & \frac{\varepsilon}{39 R^2} \Longrightarrow \gamma^* = \min \left \lbrace \frac{1}{4}, \log_{\frac{\varepsilon}{39 L_* R^2}} \left(3 \alpha \right) \right \rbrace, \text{ if } \alpha = 0, \text{ then } \gamma^* = \frac{1}{4}.
    \end{eqnarray*}
    then after (since $\frac{\mu}{4} R^2 \leqslant \frac{\varepsilon}{156}$):
    \begin{eqnarray*}
        N
        & = & \max \left\{ 0, \left\lceil 250 \left(\frac{L_*}{\mu} \right)^{1 - 2\gamma^*} \ln \left( \frac{157 (f(x^0) - f^*)}{78 \varepsilon} \right) \right\rceil \right \}
        \\
        & = & \max \left\{ 0, \left\lceil 250 \left(\frac{39 L_* R^2}{\varepsilon} \right)^{1 - 2\gamma^*} \ln \left( \frac{157 (f(x^0) - f^*)}{78 \varepsilon} \right) \right\rceil \right \}.
    \end{eqnarray*}
    iterations method will provide $f(x^N) - f^* \leqslant \varepsilon$.
\end{proof}

\end{appendix}

\end{document}